\documentclass[3p]{elsarticle}

\usepackage{lineno,hyperref}

\modulolinenumbers[5]

\journal{arXiv}

\usepackage{amsfonts}
\usepackage{amsmath, amscd,amssymb,bm,amsbsy,epsf}
    \usepackage{enumerate}
    \usepackage{paralist}
\usepackage{bbm, dsfont}
\usepackage{graphicx,color}
\usepackage{subfigure}
\usepackage{verbatim}
\usepackage{inputenc}
\usepackage{bbm}
\usepackage{multicol}
\usepackage{balance}
\usepackage{verbatim}
\usepackage{bbm, doi}

\usepackage{sfmath} 
\usepackage{sans} 

\usepackage{multirow}

\DeclareMathAlphabet{\mathpzc}{OT1}{pzc}{m}{it}

\DeclareMathOperator*{\wklim}{wk-lim}

\DeclareMathOperator*{\Div}{div}

\newtheorem{theorem}{Theorem}
\newtheorem{lemma}[theorem]{Lemma}
\newtheorem{corollary}[theorem]{Corollary}
\newtheorem{proposition}[theorem]{Proposition}
\newdefinition{definition}{Definition}
\newdefinition{hypothesis}{Hypothesis}
\newdefinition{remark}{Remark}
\newproof{proof}{Proof}
\newproof{sketch}{Sketch of the Proof}

\def\R{\bm{\mathbbm{R} } }

\def\Hdiv{\mathbf{H}_{\Div}}

\def\ind{\bm {\mathbbm{1} } }

\def\xthilde{\widetilde{\mathbf{x}}}

\def\A{{\mathcal A}}
\def\B{{\mathcal B}}
\def\C{{\mathcal C}}

\def\E{\mu}
\def\F{{\mathbf F}}
\def\Q{{\mathcal Q}}

\def\W{\mathbf W}
\def\H{\mathbf H}
\def\X{\mathbf X}
\def\Y{\mathbf Y}
\def\0{\mathbf 0}
\def\f{\mathbf f}
\def\g{\mathbf g}
\def\n{\mathbf { n}  }

\def\u{\mathbf u}
\def\v{\mathbf v}
\def\w{\mathbf w}
\def\x{\mathbf x}
\def\y{\mathbf y}
\def\div{\boldsymbol{\nabla} \cdot}
\def\grad{\boldsymbol{\nabla}}

\def\:{\negthinspace : \negthinspace}

\def\chix{{\raise.5ex\hbox{$\chi$}}}

\def\Re{\mbox{\rm I\negthinspace R}}
\def\Co{\mbox{C\kern-.47em\vrule height1.5ex width.2ex}}

\def\L2div{\mathbf{L^{2}_{div}}}
\def\H1bold{\mathbf{H^{1}}}
\def\vepsone{\mathbf{v}^{1,\epsilon}}
\def\vepstwo{\mathbf{v}^{\,2,\epsilon}}
\def\pepsone{p^{1,\epsilon}}
\def\pepstwo{p^{\,2,\epsilon}}

\def\vtaneps{\mathbf{v}_{\scriptscriptstyle T}^{\,2,\epsilon}}
\def\vnormeps{\mathbf{v}_{\scriptscriptstyle N}^{\,2,\epsilon}}
\def\wtan{\mathbf{w}_{\scriptscriptstyle T}^{2}}
\def\wnorm{\mathbf{w}_{\scriptscriptstyle N}^{2}}

\def\vtan{\mathbf{v}_{\scriptscriptstyle T}^{2}}
\def\vnorm{\mathbf{v}_{\scriptscriptstyle N}^{2}}

\def\vone{\mathbf{v}^{1}}
\def\vtwo{\mathbf{v}^{2}}
\def\pone{p^{1}}
\def\ptwo{p^{2}}
\def\veps{\mathbf{v}^{\epsilon}}
\def\peps{p^{\,\epsilon}}

\def\uone{\mathbf{u}^{1}}

\def\utan{\mathbf{u}_{\scriptscriptstyle T}^{2}}
\def\unorm{\mathbf{u}_{\scriptscriptstyle N}^{2}}

\def\symgrad{\bm{\mathcal{E} } }
\def\Hpartial{H(\partial_z, \Omega_{2})}

\def\Vomega1{\mathbf{V}_{\,\Omega_1}}

\def\xthilde{\widetilde{x}}

\def\defining{\overset{\mathbf{def}} =}

\def\L2div{\mathbf{L^{2}_{div}}}
\def\H1bold{\mathbf{H^{1}}}

\def\wone{\mathbf{w}^{1}}
\def\wtwo{\mathbf{w}^{2}}
\def\pone{p^{1}}
\def\ptwo{p^{2}}
\def\veps{\mathbf{v}^{\,\epsilon}}
\def\peps{p^{\,\epsilon}}

\def\uone{\mathbf{u}^{1}}

\def\Vomega1{\mathbf{V}_{\,\Omega_1}}

\def\xthilde{\widetilde{x}}

\begin{document}

\begin{frontmatter}

\title{A Darcy-Brinkman Model of Fractures in Porous Media}
\tnotetext[mytitlenote]{This material is based upon work supported by grant 98089 from the Department of Energy, Office of Science, USA and from project HERMES 27798 from Universidad Nacional de Colombia,
Sede Medell\'in.}


\author[mymainaddress]{Fernando A Morales}

\cortext[mycorrespondingauthor]{Corresponding Author}
\ead{famoralesj@unal.edu.co}

\address[mymainaddress]{Escuela de Matem\'aticas
Universidad Nacional de Colombia, Sede Medell\'in \\
Calle 59 A No 63-20 - Bloque 43, of 106,
Medell\'in - Colombia}

\author[mysecondaryaddress]{Ralph E Showalter}
\address[mysecondaryaddress]{Department of Mathematics, Oregon State University. Kidder Hall 368, OSU, Corvallis, OR 97331-4605}

\begin{abstract}
For a fully-coupled Darcy-Stokes system describing the exchange of
fluid and stress balance across the interface between a saturated
porous medium and an open very narrow channel, the limiting problem is
characterized as the width of the channel converges to zero. It is
proven that the limit problem is a fully-coupled system of Darcy flow
in the porous medium with Brinkman flow in tangential coordinates of
the lower dimensional interface.
\end{abstract}

\end{frontmatter}



%
%

%
%
%
%
\section{Introduction}   \label{intro}
We consider the limiting form of a system of equations describing
incompressible fluid flow in a fully-saturated region
$\Omega^\epsilon$ which consists of two parts, a porous medium
$\Omega_1$ and a very narrow channel $\Omega_{2}^\epsilon$ of width
$\epsilon > 0$ along part of its boundary, $\Gamma = \partial \Omega_1
\cap \partial \Omega_2^\epsilon$. That is, we have $\Omega^\epsilon
\equiv \Omega_{1} \cup \Gamma \cup \Omega_{2}^\epsilon$. The
filtration flow in the porous medium is governed by Darcy's law on
$\Omega_1$ and the faster flow of the fluid in the narrow open channel
by Stokes' system on $\Omega_{2}^\epsilon$.  For simplicity, we assume
that the channel is {\em flat}, that is, $\Omega_{2}^{\epsilon} \equiv
\Gamma \times (0,\epsilon)$, where $\Gamma \subset \Re^{n-1}$,
$\Re^{n-1}$ is identified with $\Re^{n-1} \times \{0\} \subset
\Re^{n}$, and $\Gamma = \partial \Omega_{1} \cap \partial
\Omega_{2}^\epsilon$ is the interface. See Figure \ref{Fig Domain
Original}. We assume that $\partial \Omega_{1} - \Gamma$ is
smooth. The Darcy and Stokes systems have very different regularity
properties, and both the tangential velocity and pressure of the fluid
are discontinuous across the interface, so the analysis is delicate.
Our goal is to establish the existence of a limit problem as the width
$\epsilon \to 0$ and to characterize it.  This limit is a
fully-coupled system consisting of Darcy flow in the porous medium
$\Omega_1$ and Brinkman flow on the part $\Gamma$ of its boundary.
%
%
\begin{figure}[h] 
        \centering
                {\resizebox{7cm}{7.8cm}
                        {\includegraphics{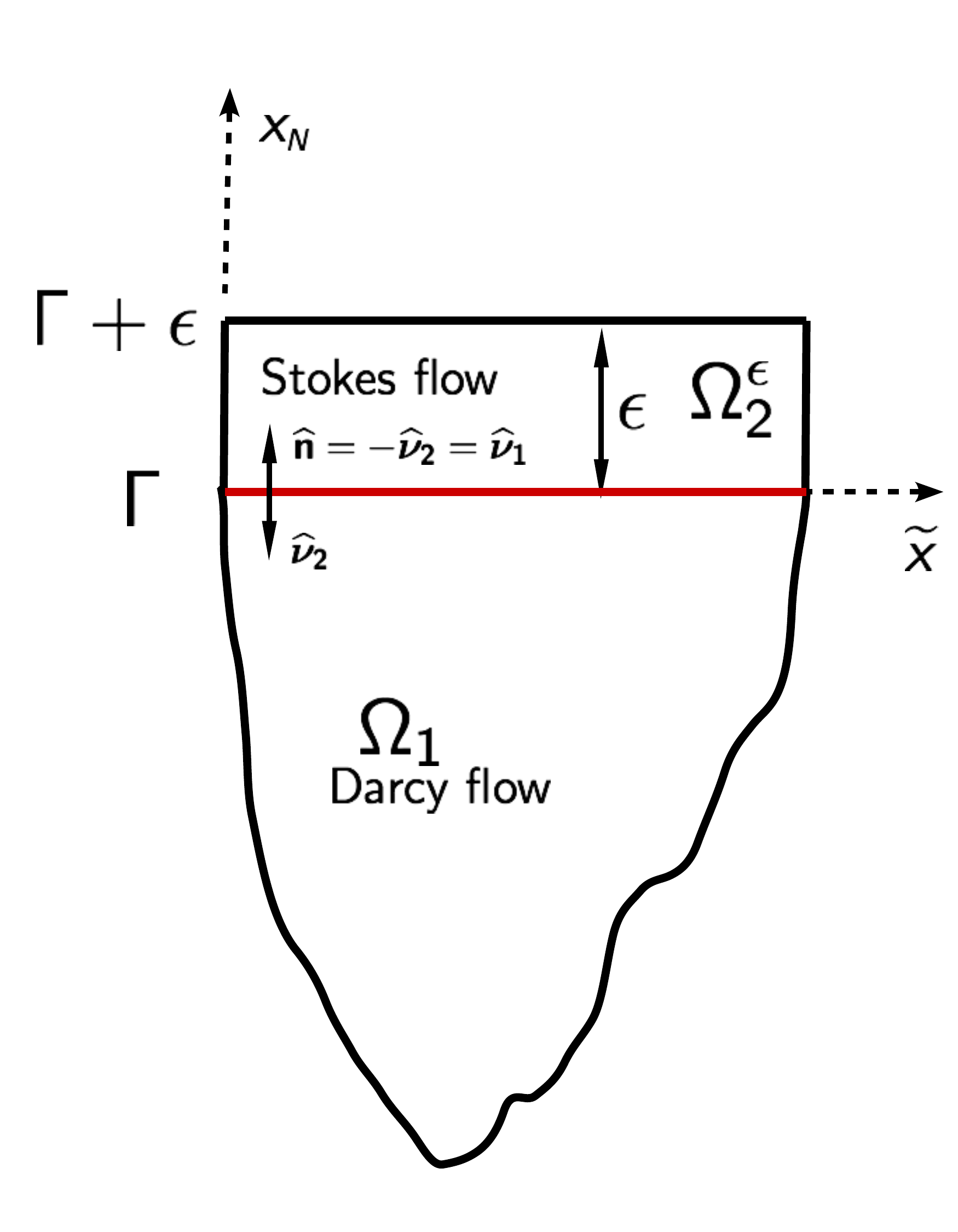} } }
        \caption{The porous medium $\Omega_1$ below the thin channel $\Omega_2^\epsilon$.}
\label{Fig Domain Original}
\end{figure}
%

%

The Darcy-Stokes system above has two types of singularities: the
geometric one coming from the narrowness of the channel
$\mathcal{O}(\epsilon)$ with respect to the dimensions, and the
physical one of high fluid flow velocity $\mathcal{O}(1/\epsilon)$ in
the channel with respect to the porous medium. These singularities
introduce multiple scales in the system which have an impact on the
numerical simulation. Some of these consequences are ill-conditioned
matrices, problems of numerical stability, poor quality of the
numerical solutions and high computational costs. On the other hand,
ignoring the presence of fractures leads to oversimplified and
unrealistic models \cite{ArbLehr, Gunzburger3}. Therefore, much
progress has been made to handle such issues from the numerical point
of view \cite{ArboBru, Gunzburger1, Gunzburger2, JaffRob05,
GaticaBabuska2010, Gatica2009, WeiMcD07}, from the analytical point of
view \cite{ArbLehr, Allaire2009, CannonMeyer71, ShowMor10, ShowMor12,
SPHung74}, from the heuristic point of view \cite{Levy83, Morales1},
and for numerical experimental coupling \cite{XieXuXue, Lamichhane,
Shavit03} by using Brinkman flow to couple numerically the Darcy and
Stokes flow models.  See especially Quarteroni et al
\cite{Quarteroni_Brinkman2011} for additional issues, references and
perspectives.

The Brinkman system has nothing to do with the usual models of porous
media flow, but rather describes Stokes flow through a sparse array of
particles for which the porosity is more than $0.8$
\cite{Brink47a,Brink47b,Nield83}. This requirement is highly restrictive since
most naturally occurring porous media have a porosity less than $0.6$.
Levy \cite{Levy81,Levy83} showed that the Brinkman system holds only
for arrays of particles whose size is precisely of order $\eta^3$,
where $\eta <<1$ is the distance between neighboring particles. Larger
particles impede the fluid flow sufficiently to be described by a
Darcy system, and smaller particles do not change the flow from the
Stokes system.  Allaire \cite{Allaire92,All97} proved and developed
this homogenization result by means of two-scale convergence.
But in the situations considered here the singular geometry of the
problem with small $\epsilon > 0$ keeps all of the fluid in the
channel very close to the interface where it is slowed by viscous
resistance forces from the porous medium. This suggests that there is
a very narrow region along the interface between Stokes flow and a
porous medium where the fluid velocity is well approximated by a
Brinkman law in the tangential coordinates. (Of course, the normal
component of velocity is determined independently by the conservation
of fluid mass across the interface.) The convergence of the
$\epsilon$-model established below provides an explanation for the
success of numerical approximations that use an intermediate Brinkman
system to connect Darcy and Stokes flows across an interface by
adjusting the coefficients.

Our model describes two fundamental situations. The first is the rapid
tangential flow near the boundary of a porous medium where the
porosity becomes large due to the inefficiency of the packing of the
particles of the medium. If the particles in this {\em boundary
channel} are sufficiently sparse, the less impeded flow begins to
follow this Stokes-like model in the substantial space between
particles. See Nield \& Bejan \cite{NieldBejan} for additional
discussion and perspectives.
The second and more common situation is obtained by reflecting
$\Omega^\epsilon$ about the outer wall of the channel,
$\Gamma\times\{\,\epsilon\,\}$. This provides a model for a narrow
{\em interior fracture} of width $2\,\epsilon$ in a porous medium.
(See Remark~\ref{rem interface options}.) Such a fracture is assumed
to be open, so fluid flow follows the Stokes system; debris-filled
fractures have been modeled as regions of Darcy flow with very high
permeability \cite{CannonMeyer71,JaffRob05,ShowMor10, ShowMor12}.  In
the limiting problem below, the fracture is described by Brinkman flow
in tangential coordinates coupled on {\em both} sides to the
surrounding Darcy flow of the porous medium.

In this work we present the full asymptotic analysis for this coupled
Darcy-Stokes system in order to derive a new model, free of
singularities. The {\em limit problem} consists of a Darcy-Brinkman
fully coupled system with Darcy flow on the original porous medium and
Brinkman flow on the surface approximating the adjacent channel or
internal fracture; see Figure \ref{Fig Domain Scaling}. The spaces of
convergence will be found and the convergence of solutions will be
extablished. It is worthwhile to stress that the method is remarkably
simple with respect to other techniques as it uses only scaling,
standard weak convergence methods and general Hilbert space theory. It
is precisely this {\em simplicity} that gives the method its power and
success in handling simultansously the asymptotic analysis, the
multiple scales and the substantially different structures of Darcy
and Stokes systems. In particular, we obtain explicitly the
correspondence between the coefficients in the Beavers-Joseph-Saffman
interface condition and those in the limiting Brinkman system.

\subsection*{Notation}
We shall use standard function spaces (see \cite{Temam79, Adams}).
For any smooth bounded region $G$ in $\R^N$ with boundary $\partial
G$, the space of square integrable functions is denoted by $L^{2}(G)$,
and the Sobolev space $H^{1} (G) $ consists of those functions in
$L^2(G)$ for which each of the first-order weak partial derivatives
belongs to $L^2(G)$. The {\em trace} is the continuous linear
function $\gamma:H^{1}(G) \rightarrow L^{2}(\partial G)$ which agrees
with restriction to the boundary, i.e., $\gamma(w) = w \big\vert_{\partial
G}$ on smooth functions. Its kernel is $H_{0}^{1}(G) \defining \{w \in
H^{1}(G): \gamma(w) = 0\}$.  The trace space is $H^{1/2}(\partial G)
\defining \gamma(H^{1}(G))$, the range of $\gamma$ endowed with the
usual norm from the quotient space $H^{1}(G)/H_{0}^{1}(G)$, and we
denote by $H^{-1/2}(\partial G)$ its topological dual.
Column vectors and corresponding vector-valued functions will be
denoted by boldface symbols, e.g., we denote the product space
$\big[L^2(G)\big]^N$ by $\mathbf{L}^{2}(G)$ and the respective N-tuple
of Sobolev spaces by $\mathbf{H}^{1}(G) \defining
\big[H^{1}(G)\big]^N$.  Each $w \in H^{1}(G)$ has {\em gradient}
$\grad w = \big(\frac{\partial w}{\partial x_{1}}, \ldots,
\frac{\partial w}{\partial x_{\scriptscriptstyle N}} \big) \in
\mathbf{L}^{2}(G)$.
We shall also use the space $\Hdiv(G)$ of vector functions $\w \in
{\mathbf L}^2(G)$ whose weak divergence $\div \w$ belongs to
$L^{2}(G)$.
Let $\n$ be the unit outward normal vector on $\partial G$.  If $\w$
is a vector function on $\partial G$, we denote its normal component
by $w_{n} = \gamma(\w)\cdot\n$ and the normal projection by $w_{n}
\n$. The tangential component is $\wtan = \w - w_{n}\, {\n}$.  For the
functions $\w \in \Hdiv(G)$, there is a {\em normal trace} defined on
the boundary values, which will be denoted by $\w \cdot \n \in
H^{-1/2}(\partial G)$. For those $\w \in \mathbf{H}^{1}(G)$ this
agrees with $\gamma(\w) \cdot \n$.
Greek letters are used to denote general second-order tensors.  The
contraction of two tensors is given by $\sigma : \tau = \sum_{i, \,
j}\, \sigma_{ij}\tau_{ij}$.
For a tensor-valued function $\tau$ on $\partial G$, we denote the
normal component (vector) by $\tau(\n) \defining \sum_{j} \,
\tau_{ij}\, \n_{j} \in \R^N$, and its normal and tangential parts by
$\big(\tau(\n)\big)\cdot\n = \tau_{n} \defining \sum_{i, \, j} \,
\tau_{ij} \n_{i} \n_{j}$ and $\ {\mathbf \tau}(\n)_{\scriptscriptstyle T}
\defining \tau (\n) - \tau_{n} \n$, respectively.
For a vector function $\w \in \mathbf{H}^{1}(G)$, the tensor $(\grad \w)_{ij} =
\frac{\partial w_i}{\partial x_j}$ is the {\em gradient} of $\w$ and
$\big(\symgrad(\w)\big)_{i j} = \tfrac{1}{2} \big(\frac{\partial w_i}{\partial x_j} +
\frac{\partial w_j}{\partial x_i} \big)$ is the {\em symmetric
gradient}.

Next we describe the geometry of the domains to be used in the present
work; see Figure \ref{Fig Domain Original} for the case $N = 2$.
The disjoint bounded domains $\Omega_{1}$ and $\Omega_{2}^{\epsilon}$
in $\R^{N}$ share the common {\em interface}, $\Gamma \defining
\partial \Omega_{1} \cap \partial \Omega_{2}^{\epsilon} \subset
\R^{N-1} \times \{0\}$, and we define $\Omega^\epsilon \defining
\Omega_1 \cup \Gamma \cup \Omega_2^\epsilon$.  For simplicity of
notation we have assumed that the interface is flat and, moreover,
that the domain $\Omega_{2}^{\epsilon}$ is a cylinder:
$\Omega_{2}^{\epsilon}\defining \Gamma\times (0,\epsilon)$.  We denote
by $\n(\cdot)$ the unit outward normal vector on $\partial \Omega_{1}$
and on $\partial \Omega_{2}^{\epsilon} - \Gamma$.  The domain
$\Omega_{1}$ is the porous medium, and $\Omega_{2}^{\epsilon}$ is the
free fluid region. We focus on the case where $\Omega_{2}^{\epsilon}$
is the lower half of a symmetric narrow horizontal fracture of width
$\epsilon$, $0 < \epsilon \ll 1$, and $\Omega_{1}$ is the porous
medium below the fracture. By modifying boundary conditions on $\Gamma
+ \epsilon$, we recover the case of a free-fluid region adjacent to (a
flat part of) $\partial \Omega_1$.

For a column vector $\x =
\big(x_{\,1},\,\ldots,\,x_{\scriptscriptstyle N -
1},\,x_{\scriptscriptstyle N}\big)\in \R^{N}$ we denote the
corresponding vector in $\R^{N-1}$ consisting of the first $N-1$
components by $\widetilde{\x}=
\big(x_{1},\,\ldots,\,x_{\scriptscriptstyle N - 1}\big)$, and we
identify $\R^{N-1} \times \{0\}$ with $\R^{N-1}$ by $\x =
(\widetilde{\x},x_{\scriptscriptstyle N})$.
For a vector function $\w$ on $\Gamma$ we see $\w_{\scriptscriptstyle
T} = \widetilde{\w}$ is the first $N-1$ components and
$\w_{\n} = w_{\scriptscriptstyle N}$ is the last
component of the function.  The operators $\grad_{\!T}$,
$\grad_{\!T}\boldsymbol{\cdot}$ denote respectively the
$\R^{N-1}$-gradient and the $\R^{N-1}$-divergence in directions
tangent to $\Gamma$, i.e. $\grad_{\!T} = \big(\frac{\partial}{\partial
x_{1}}, \ldots, \frac{\partial}{\partial x_{\scriptscriptstyle N - 1}}
\big)$, $\grad_{\!T}\boldsymbol{\cdot} = \big(\frac{\partial}{\partial
x_{1}}, \ldots, \frac{\partial}{\partial x_{\scriptscriptstyle N - 1}}
\big) \boldsymbol{\cdot}$.
%

\subsection{The Equations}
We determine the fluid flow through the porous medium $\Omega_{1}$ by
the {\em Darcy system}, i.e.
\begin{subequations}\label{darcy_sys} 
\begin{equation}
\div \vone = h_1 \,,
\end{equation}
\begin{equation}\label{darcy}
\Q \,\vone +  \grad \pone = \0 \,, \quad \text{in} \, \Omega_{1} .
\end{equation}
\end{subequations}
The functions $\pone$, $\vone$ are respectively, the {\em pressure} and {\em
filtration velocity} of the incompressible viscous fluid in the pores.
The resistance tensor $\Q$ is the shear viscosity $\mu$ of the fluid times
the reciprocal of the {\em permeability} of the structure.
The flow of the fluid in the adjacent open channel
$\Omega_{2}^{\epsilon}$ is described by the {\em Stokes system}
\cite{Temam79, SP80}
\begin{subequations}\label{Stokes system}
\begin{equation}\label{Eq Divergence free condition}
\grad\cdot \vtwo = 0 ,
\end{equation}
\begin{equation}\label{Momentum fluid}
-\grad\cdot\sigma^{2}+\grad \ptwo = \f_{2} ,
\end{equation}
\begin{equation}\label{Newtonian fluid}
\sigma^{2} = 2\,\epsilon\,\E\, \symgrad(\vtwo)\,,
\quad\text{in}\;\Omega_{2}^{\epsilon} .
\end{equation}
\end{subequations}
Here, $\vtwo$, $\sigma^{2}$, $\ptwo$ are respectively, the {\em
velocity}, \textit{stress tensor} and the \textit{pressure} of the
fluid in $\Omega_{2}^{\epsilon}$, while $\symgrad(\vtwo)$ denotes the
symmetric gradient of the velocity field. Among the equations above,
only \eqref{darcy} and \eqref{Newtonian fluid} are constitutive and
subject to scaling. Darcy's law \eqref{darcy} describes the fluid on the
part of the domain with fixed geometry, hence, it is not scaled. The
law \eqref{Newtonian fluid} establishes the relationship between the
strain rate and the stress for the fluid in the thin channel,
therefore it is scaled according to the geometry. Finally, recalling
that $\grad\cdot \vtwo = 0$, we have
\begin{equation}\label{Eq stress divergence to laplacian velocity}
\grad\cdot\sigma^{2} = 
\grad\cdot \big[2\,\epsilon\,\E\,\symgrad\big(\vtwo\big) \big]
= \epsilon\,\E \grad\cdot\grad\vtwo .
\end{equation}
This observation transforms the system \eqref{Eq Divergence free
condition}, \eqref{Momentum fluid} and \eqref{Newtonian fluid} to the
classical form of {\em Stokes flow} system.

%
\subsection{Interface Conditions}
The interface coupling conditions account for the stress and mass
conservation.  For the stress balance, the tangential and normal
components are given by the {\em Beavers-Joseph-Saffman}
\eqref{Beavers - Joseph -Saffman condition} and the classical {\em
Robin} boundary condition \eqref{Eq Interface Normal Stress Balance}
respectively, i.e.
\begin{subequations}\label{Eq Interface Stress Balance}
\begin{equation}\label{Beavers - Joseph -Saffman condition}
\sigma_{\scriptscriptstyle T}^{2} =  
\epsilon^{2}\,\beta \, \sqrt{\Q}\, \vtan \,,
\end{equation}
\begin{equation}\label{Eq Interface Normal Stress Balance}
\sigma_{\n}^{2} - \ptwo + \pone 
= \alpha\, \vone \cdot \n   \text{ on } \Gamma .
\end{equation}
In the expression \eqref{Beavers - Joseph -Saffman condition} above,
$\epsilon^{2}$ is a scaling factor destined to balance out the
geometric singularity introduced by the thinness of the channel.  In
addition, the coefficient $\alpha \ge 0$ in \eqref{Eq Interface Normal
Stress Balance} is the {\em fluid entry resistance}.
\end{subequations}
In the present work it is assumed that the velocity is curl-free on
the interface, so the conditions \eqref{Beavers - Joseph
-Saffman condition} and \eqref{Eq Interface Normal Stress Balance}
are equivalent to 
\begin{subequations}\label{EQ interface conditions}
\begin{equation}\label{Beavers - Joseph condition}
\epsilon\,\mu\, \frac{\partial}{\partial\, \n}\,\vtan 
= \epsilon\,\mu\,\frac{\partial}{\partial\, x_{\scriptscriptstyle N}}\,\vtan 
= \epsilon^{2} \beta \,  \sqrt{\Q}\, \vtan \,,
\end{equation}
\begin{equation}  \label{scaled interface normal stress}
\epsilon\,\mu\Big(\frac{\partial \,\vtwo}{\partial\,\n}\cdot\n\Big)-p^{2}+p^{1}
=\epsilon\,\mu\, \frac{\partial \,\vnorm}{\partial\,x_{\scriptscriptstyle N}}-p^{2}+p^{1}
=\alpha\,\vone\cdot\n \text{ on } \Gamma .
\end{equation}
The conservation of fluid across the interface gives the normal fluid
flow balance
\begin{equation}\label{admissability constraint}
\vone \cdot \n = \vtwo \cdot \n  \text{ on } \Gamma .
\end{equation}
\end{subequations}
The {\em interface conditions} \eqref{EQ interface conditions}
will suffice precisely to
couple the Darcy system \eqref{darcy_sys} in $\Omega_{1}$ to the
Stokes system \eqref{Stokes system} in $\Omega_{2}^{\epsilon}$.
%
%
\subsection{Boundary Conditions}
We choose the {\em boundary conditions} on $\partial\,
\Omega^{\epsilon} = \partial \Omega_{1} \cup \partial
\Omega_{2}^{\epsilon} - \Gamma$ in a classical simple form, since they
play no essential role here.  On the exterior boundary of the porous
medium, $\partial \Omega_{1} - \Gamma$, we impose the {\em drained}
conditions,
\begin{subequations}\label{BC Top boundary condition}
\begin{equation}\label{BC Drained}
\pone = 0 \, \quad \text{on }\, \partial \Omega_{1} - \Gamma .
\end{equation}
As for the exterior boundary of the free fluid, $\partial \,\Omega_{2}^\epsilon
- \Gamma$, we choose no-slip conditions on the wall of the cylinder,
\begin{equation}\label{BC Null Flux}
\vtwo=0 \, \quad \text{on }\, \partial \Gamma\times(0,\epsilon) .
\end{equation}
On the top of the cylinder
$\Gamma+\epsilon\defining\big\{(\widetilde{x},
\epsilon):\widetilde{x}\in \Gamma\big\}$, the hyper-plane of symmetry,
we have mixed boundary conditions, a Neumann-type condition on the
tangential component of the normal stress
\begin{equation}\label{BC tangential boundary condition}
\frac{\partial \,\vtwo}{\partial\,\n}
- \Big(\frac{\partial \,\vtwo}{\partial\,\n}\cdot\n\Big)\n
=\frac{\partial \,\vtan}{\partial\,x_{\scriptscriptstyle N}} 
= 0\quad\text{on}\;\Gamma+\epsilon ,
\end{equation}
and a null normal flux condition, i.e. 
\begin{equation}\label{BC normal boundary condition}
\vtwo\cdot\n =
\vnorm = 0 \quad\text{on}\;\Gamma+\epsilon .
\end{equation}
\end{subequations}
\begin{remark}  \label{rem interface options}
The boundary conditions \eqref{BC Null Flux} and \eqref{BC tangential
boundary condition} are appropriate for the mid-line of an {\em
internal fracture} with symmetric geometry. In that case, the
interface conditions \eqref{EQ interface conditions} hold on both
sides of the fracture. If $\Omega_{2}^\epsilon$ is an adjacent open
channel along the boundary of $\Omega_1$, then we extend the no-slip
condition \eqref{BC Null Flux} to hold on all of $\partial
\,\Omega_{2}^\epsilon - \Gamma$.
\end{remark}
\begin{remark}\label{Rem Scaling Coments}
For a detailed exposition on the system's adopted scaling namely, the
fluid stress tensor \eqref{Newtonian fluid} and the
Beavers-Joseph-Saffman condition \eqref{Beavers - Joseph -Saffman
condition}, together with the \textbf{formal} asymptotic analysis see
\cite{Morales1}.
\end{remark}
%

\subsection*{Preliminary Results}
We  close this section by recalling some classic results.
\begin{lemma}\label{Th Surjectiveness from Hdiv to H^1/2}
Let $G \subset \R^{N}$ be an open set with Lipschitz boundary, let
$\n$ be the unit outward normal vector on $\partial G$.  The normal
trace operator $\u \in \Hdiv(G) \mapsto \u \cdot \n \in
H^{-1/2}(\partial G)$ is defined by
\begin{equation}\label{Eq Normal Trace Definition}
\big\langle \u\cdot \n, \phi \big\rangle
_{\scriptscriptstyle H^{-1/2}(\partial G), \, H^{1/2}(\partial G)}
\defining 
\int_{G} \Big(\u\cdot \grad \phi + \div \u \, \phi \big)\, dx ,
\quad \phi \in H^{1}(G).
\end{equation}
For any $g\in H^{-1/2}(\partial G)$ there exists $\u\in \Hdiv(G)$ such
that $\u\cdot \n = g$ on $\partial G$ and $\Vert \u \Vert_{\Hdiv
(G)}\leq K \Vert g \Vert_{H^{-1/2}(\partial G)}$, with $K$ depending
only on the domain $G$. In particular, if $g$ belongs to
$L^{2}(\partial G)$, the function $\u$ satisfies the estimate $\Vert \u
\Vert_{\Hdiv (G)}\leq K \Vert g \Vert_{0,\partial G}$.
\end{lemma}
\begin{proof}
See Lemma 20.2 in \cite{TartarSobolev}.
\qed
\end{proof}

We shall recall in Section 2 that the boundary-value problem
consisting of the Darcy system \eqref{darcy_sys}, the Stokes system
\eqref{Stokes system}, the interface coupling conditions \eqref{EQ
interface conditions} and the boundary conditions \eqref{BC Top
boundary condition} can be formulated as a constrained minimization
problem.  Let $\X$ and $\Y$ be Hilbert spaces and let $\A:
\X\rightarrow \X '$, $\B: \X\rightarrow \Y '$ and $\C: \Y\rightarrow
\Y '$ be continuous linear operators. The problem is to find a pair satisfying
\begin{equation}\label{Pblm operators abstrac system}
\begin{split}
(\x, \y)\in \X\times \Y: \quad 
\A \x + \B '\y  = F_{1}\quad \text{in}\; \X ' , 
\\
- \B \x  + \C \y = F_{2} \quad \text{in}\; \Y ' 
\end{split}
\end{equation}
with $F_{1}\in \X '$ and $F_{2} \in \Y '$.  We present a well-known
result \cite{GirRav79} to be used in this work.
\begin{theorem}\label{Th well posedeness mixed formulation classic}
Assume that the linear operators $\A: \X\rightarrow \X'$, $\B:
\X\rightarrow \Y '$, $\C: \Y\rightarrow \Y '$ are continuous and
\begin{enumerate}[(i)]
\item $\A$ is non-negative and $\X$-coercive on $\ker (\B)$,

\item $\B$ satisfies the inf-sup condition 
\begin{equation}\label{Ineq general inf-sup condition}
   \inf_{\y \, \in \, \Y} \sup_{\x \, \in \, \X}
   \frac{\vert  \B\x(\y) \vert }{\Vert \x\Vert_{\X}\, \Vert \y \Vert_{\Y}}  >0\,,
\end{equation}

\item $C$ is non-negative and symmetric.
\end{enumerate}
Then, for every $F_{1} \in \X '$ and $F_{2} \in \Y '$ the problem
\eqref{Pblm operators abstrac system}
has a unique solution $(\x, \y)\in \X \times \Y$, and it satisfies the
estimate
\begin{equation} \label{mix-est}
\Vert\x\Vert_{\X} + \Vert \y\Vert_{\Y} \leq c\, (\Vert F_{1}\Vert_{\X '} 
+ \Vert F_{2}\Vert_{\Y '})
\end{equation}
for a positive constant $c$ depending only on the preceding
assumptions on $\A$, $\B$, and $\C$.
\end{theorem}
Several variations of such systems have been extensively developed,
e.g., see \cite{Show10} for nonlinear degenerate and time-dependent
cases.

\section{A well-posed Formulation}  \label{sec-wellposed}
%
In this section we present a mixed formulation for the problem on the
domain $\Omega^\epsilon$ described in Section \ref{intro} and show it
is well-posed. In order to remove the dependence of the domain
$\Omega^\epsilon$ on the parameter $\epsilon > 0$, we rescale
$\Omega_2^\epsilon$ and get an equivalent problem on the domain
$\Omega^1$.  

The abstract problem is built on the function spaces
\begin{subequations}\label{Def Function Spaces}
\begin{equation}\label{Def Velocities Two}
\X_{2}^{\epsilon}\defining
\big\{\v\in\H1bold(\Omega_{2}^{\epsilon}):\v=0 \; \text{on}\,
\partial\,\Gamma\times(0,\epsilon),\,\v\cdot\n = 0\;\text{on}\,\Gamma+\epsilon\, \big\} ,
\end{equation}
\begin{equation}\label{Def Velocities Match}
\X^{\epsilon} \defining \big\{[\,\v^{1},\v^{2}\,]\in
\Hdiv(\Omega_{1})\times\X_{2}^{\epsilon}:
\v^{1}\cdot\n=\v^{2}\cdot\n
\;\text{on}\,\Gamma\big\} 
=
\big\{ \v \in \Hdiv(\Omega^\epsilon): \v^{2} \in \X_2^\epsilon \big\},
\end{equation}
\begin{equation}\label{Def Pressures}
\Y^{\epsilon}\defining
L^{2}(\Omega^{\epsilon}) ,
\end{equation}
\end{subequations}
endowed with
their respective natural norms. We shall use the following hypothesis.

\begin{hypothesis}\label{Hyp Bounds on the Coefficients}
It will be assumed that $\mu > 0$ and the coefficients $\beta$ and
$\alpha$ are nonnegative and bounded almost everywhere. Moreover, the
tensor $\Q$ is elliptic, i.e., there exists a $C_{\Q}> 0$ such that $(\Q
\,\x)\cdot \x \geq C_{\Q} \Vert \x\Vert^{2}$ for all $\x\in \R^{N}$.
\end{hypothesis}
\begin{proposition}\label{Th Weak Variational Formulation}
The boundary-value problem consisting of the equations \eqref{darcy_sys},
\eqref{Stokes system}, the interface coupling conditions \eqref{EQ
interface conditions} and the boundary conditions \eqref{BC Top
boundary condition} has the constrained variational formulation
\begin{subequations}\label{scaled problem}
\begin{flushleft}
$\big[\,\v^{\epsilon},p^{\epsilon}\,\big]\in\X^{\epsilon}\times\Y^{\epsilon}: $
\end{flushleft}
\vspace{-0.4cm}
\begin{align}
\label{scaled problem 1}
\int_{\Omega_1} \big( \Q \,\v^{1,\,\epsilon} \cdot \w^1 & 
-
\pepsone\,\grad\cdot\wone
\big)\,dx
+ \int_{\Omega_2^{\epsilon}} \big(\,\epsilon\,\E
\grad\,\vepstwo - \pepstwo
\delta\ \big)\: \grad\wtwo
 \,d\widetilde{x} dx_{\scriptscriptstyle N}
\\
+\,\alpha\int_{\Gamma}\big(\vepstwo\,\cdot\n\big)\, & \big(\wtwo\cdot\n\big)\,dS 
+ \int_{\Gamma} \epsilon^{2} \, \beta \, \sqrt{\Q}
\;\vtaneps \cdot \wtan \,d S 
= \int_{\Omega_2^{\epsilon}} {\mathbf f^{\,2,\,\epsilon}} \cdot
\wtwo \,d\widetilde{x} \,dx_{\scriptscriptstyle N} ,
\nonumber
\\
\label{scaled problem 2}
\int_{\Omega_1} \grad\cdot\vepsone\,
\varphi^1 \,dx &
+ \int_{\Omega_2^{\epsilon}}  \grad\cdot\vepstwo\, \varphi^2
 \,d\widetilde{x} \,dx_{\scriptscriptstyle N}
= \int_{\Omega_1} h^{1,\,\epsilon} \, \varphi^1 \, dx ,
\end{align}
\end{subequations}
\flushright{for all $\big[\w,\varphi\big]\in\X^{\epsilon}\times\Y^{\epsilon}$.}
\end{proposition}

\begin{proof}
Let $\v^{\epsilon}=\big[\vepsone,\vepstwo\big]$,
$p^{\epsilon}=\big[\pepsone,\pepstwo\big]$ be a solution and choose a
test function $\w = [\wone,\wtwo] \in \X^\epsilon$.  Substitute the
relationship \eqref{Eq stress divergence to laplacian velocity} in the
momentum equation \eqref{Momentum fluid} and multiply the outcome by
$\wtwo$. Multiply the Darcy law \eqref{darcy} by $\wone$. Integrating
both expressions and adding them together, we obtain
\begin{multline}\label{Eq First Integration for Velocities}
\int_{\Omega_1}\Big(\,\Q\,\vepsone\cdot\w^{1}
-\pepsone\,\delta\:\symgrad\big(\w^{1}\big)\,\Big)\,dx
+\int_{\Omega_2^{\epsilon}}\big(\epsilon\,\E\grad\vtwo
-\pepstwo\,\delta\big)\:\grad\w^{2}\,dx\\
+\int_{\Gamma} \Big(\pepsone\,\n\cdot\wone
+\epsilon\big(\grad\,\vepstwo\,(\n) \big)\cdot\wtwo
-\pepstwo\,\big(\wtwo\cdot\n\big)\Big)\,dS
= \int_{\Omega_2^{\epsilon}} \mathbf{f}^{2}\cdot\wtwo\,dx .
\end{multline}
Since $\w$ satisfies the admissibility constraint
\eqref{admissability constraint}, $\wone\cdot\n=\wtwo\cdot\n$ on
$\Gamma$, the interface integral reduces to
\begin{equation*}
\int_{\Gamma} \Big(\epsilon\,\frac{\partial\,\vepstwo}{\partial\,\n}\cdot \wtwo
+\big(\pepsone -\pepstwo\big)\,\big(\wone\cdot\n\big)\Big)\,dS .
\end{equation*}
Decomposing the velocity terms into their normal and
tangential components, we obtain
\begin{equation*}
\int_{\Gamma}
\Big\{\epsilon\Big(\frac{\partial\,\vepstwo}{\partial\,\n}\Big)_{T}\cdot\wtan
+\Big(\epsilon\Big(\frac{\partial\,\vepstwo}{\partial\,\n}\cdot\n\Big)
+\pepsone -\pepstwo\Big)\,\big(\wtwo\cdot\n\big)\Big\}\,dS.
\end{equation*}
Therefore, the interface conditions \eqref{Beavers - Joseph condition}
and \eqref{scaled interface normal stress} yield
\begin{equation*}
\int_{\Gamma}
\epsilon^{2}\,\beta\,\sqrt{\Q}\,\vtaneps\cdot\wtan\,dS
+\,\alpha\int_{\Gamma}\big(\vepsone\cdot\n\big)\,\big(\wone\cdot\n\big)\,dS ,
\end{equation*}
and inserting this in \eqref{Eq First Integration for Velocities}
yields the variational statement \eqref{scaled problem 1}.  Next,
multiply the fluid conservation equations with a test function $\varphi =
[\varphi^{1}, \varphi^{2}] \in L^{2}(\Omega^{\epsilon})$, integrate
over the corresponding regions and add them together to obtain the
variational statement \eqref{scaled problem 2}.  Conversely, by making
appropriate choices of test functions in \eqref{scaled problem} and
reversing the preceding calculations, it follows that these
formulations are equivalent.
\qed
\end{proof}
%
\subsection{The mixed formulation}
Define the operators $A^{\epsilon}: \X^{\epsilon}\rightarrow
(\X^{\epsilon})'$, $B^{\epsilon}: \X^{\epsilon}\rightarrow
(\Y^{\epsilon})'$ by
\begin{subequations}
\begin{multline}    \label{def-A}
A^\epsilon\v(\w) \defining 
\int_{\Omega_1} \big( \Q \,\v^{1} \cdot \w^1 \big)\,dx
+ \alpha\int_{\Gamma}\big(\vone\,\cdot\n\big)\,  \big(\wone\cdot\n\big)\,dS 
\\
+  \int_{\Gamma} \epsilon^{2} \, \beta \, \sqrt{\Q} \vtan \cdot \wtan \,dS
+ \int_{\Omega_2^{\epsilon}} \big(\,\epsilon\,\E
\grad\,\vtwo \: \grad\wtwo
 \,d\widetilde{x} dx_{\scriptscriptstyle N}\,,
\end{multline}
\begin{equation}    \label{def-B}
B^\epsilon\v(\varphi) \defining 
\int_{\Omega_1} \grad\cdot\vone\, \varphi^1 \,dx 
+ \int_{\Omega_2^{\epsilon}}  \grad\cdot\vtwo\, \varphi^2
 \,d\widetilde{x} \,dx_{\scriptscriptstyle N},
\end{equation}
\begin{flushright}
 for all  $\v, \w \in \X^{\epsilon},\, \varphi \in \Y^{\epsilon} $. 
\end{flushright}
\end{subequations}
These are denoted also by matrix operators
\begin{subequations}\label{Def Epsilon Mixed Operators}
\begin{equation}\label{operator A}
A^{\epsilon} =
\left(\begin{array}{cc}\, \Q + \gamma_{\, \n} '\,\alpha\,\gamma_{\, \n}   & 
\0 \\[3pt]
\0 & 
\epsilon^{2}\,\gamma_{\,T}'\,\beta\,\sqrt{\Q}\,\gamma_{\,T}
+\,\epsilon\,(\grad)'\,\E\,\grad\end{array} \right) 
\end{equation}
and
\begin{equation}\label{operator B }
B^{\epsilon} =
\left(\begin{array}{cc} \,\grad\,\cdot & 0 \\
0 &  \grad\,\cdot \end{array} \right)
= \left(\begin{array}{cc} \,\text{div} & 0 \\
0 &  \text{div} \end{array} \right) .
\end{equation}
\end{subequations}
With these operators, the variational formulation \eqref{scaled
problem} for the boundary-value problem takes the form
\begin{equation}\label{Def Scaled Problem MIxed Formulation} 
\begin{split} 
%
[\,\veps,\peps\,]\in \X^{\epsilon}\times\Y^{\epsilon}:
A^{\epsilon}\,\veps- (B^{\epsilon}) '\,\peps&=\mathbf{f}^{2,\,\epsilon} , 
\\
%
 B^{\epsilon}\,\veps &= h^{1,\,\epsilon} . 
%
\end{split}
\end{equation}
Here, the unknowns are $\veps \defining [\,\vepsone,\vepstwo\,] \in
\X^{\epsilon}$, $\peps\defining[\,\pepsone, \pepstwo\,]\in
\Y^{\epsilon}$. Next, we show that the Problem \eqref{Def Scaled
Problem MIxed Formulation} is well-posed by verifying that the hypotheses
of Theorem \ref{Th well posedeness mixed formulation classic} are
satisfied.
\begin{lemma}\label{Th coercivity of A epsilon}
The operator $A^{\epsilon}$ is $\X^{\epsilon}$-coercive
over $\X^{\epsilon}\cap \ker (B)$.
\end{lemma}
\begin{proof} The form
$\displaystyle
A^{\epsilon}\,\v\,(\v)+\int_{\Omega_1}\big(\grad\cdot\v\big)^{2}$ is
$\X^{\epsilon}$-coercive, and $\grad \cdot \v\,\vert_{\,\Omega_1}=0$
whenever $\v \in \ker(B)$.  
\qed
\end{proof}

In order to verify the inf-sup condition for the operator
$B^{\epsilon}$ we introduce the space
\begin{equation}\label{Def Boundary Space}
\F(\Omega^{\epsilon})\defining
\big\{\v\in\H1bold(\Omega^{\epsilon}):\v=0 \; \text{on}\,
\partial \Omega_{2}^{\epsilon} - \Gamma \big\} ,
\end{equation}
endowed with the $\H1bold(\Omega^{\epsilon})$-norm.
\begin{lemma}\label{Th B closed range epsilon}
The operator $B^{\epsilon}$ has closed range.
\end{lemma}
\begin{proof} 
Since $\F(\Omega^{\epsilon})\subseteq \X^{\epsilon}$ and the
Poincar\'e inequality gives a constant $C>0$ such that $\Vert \v
\Vert_{\X}\leq C \Vert\v\Vert_{\,\mathbf{H}^{1}(\Omega^{\epsilon})}$
for all $\v\in \F(\Omega^{\epsilon})$, we have
\begin{multline}\label{Ineq Inf Sup for B epsilon}
\inf_{\varphi\,\in \,L^{2}(\Omega^{\epsilon})}\,\sup_{\v\,\in\,\X}
\frac{B^{\epsilon}\,\v\,(\varphi)}
{\Vert \v \Vert_{\X} \Vert \varphi \Vert_{L^2(\Omega^{\epsilon})}}
\geq
\inf_{\varphi\,\in \,L^{2}(\Omega^{\epsilon})}\,\sup_{\v\,\in\,\F(\Omega^{\epsilon})}
\frac{B^{\epsilon}\,\v\,(\varphi)}
{\Vert \v \Vert_{\X} \Vert \varphi \Vert_{L^2(\Omega^{\epsilon})}}
\\
\geq
\frac{1}{C}\,
\inf_{\varphi\,\in \, L^{2}(\Omega^{\epsilon})}\,
\sup_{\v\,\in\,\F(\Omega^{\epsilon})}
\frac{B^{\epsilon}\,\v\,(\varphi)}
{\Vert\v\Vert_{\mathbf{H}_{0}^{1}(\Omega^{\epsilon})} \Vert \varphi \Vert_{L^2(\Omega^{\epsilon})}}.
\end{multline}
The last term above is known to be positive (see Theorem 3.7 in
\cite{GirRav79}), since it corresponds to the inf-sup condition
for the Stokes problem with mixed boundary conditions:
\begin{align}\label{Pblm Stokes Mixed Bd conditions}
& -\div \big(\E \, \grad \w\big) + \grad q = \g_{1}\, , &
& \div \w = 0 \,   \text{ in }\, \Omega^{\epsilon},\\
& \w = 0\; \text{ on }\, \partial \Omega_{2}^{\epsilon} - \Gamma \, , & 
& \E\, \frac{\partial \w}{\partial \n} - q \, \n = \g_{2} \; 
\text{ on }\, \partial \Omega^{\epsilon} - \partial \Omega_{2}^{\epsilon} ,
\end{align}
and forcing terms $\g_{1}$ and $\g_{2}$ satisfying the necessary
hypotheses of duality.
\qed
\end{proof}
\begin{theorem}\label{Th well-posedness of the epsilon problem}
The Problem \eqref{Def Scaled Problem MIxed Formulation}  is well-posed.
\end{theorem}
\begin{proof}
Due to Lemmas \ref{Th coercivity of A epsilon} and \ref{Th B closed range epsilon} above, the operators $A^{\epsilon}$ and $B^{\epsilon}$ satisfy the hypotheses of Theorem \ref{Th well posedeness mixed formulation classic} and the result follows.
\qed
\end{proof}
%
%
%
%
%
%
%
%
\subsection{The Reference Domain}
\begin{figure}[h] 
	\centering
		{\resizebox{7cm}{8cm}
			{\includegraphics{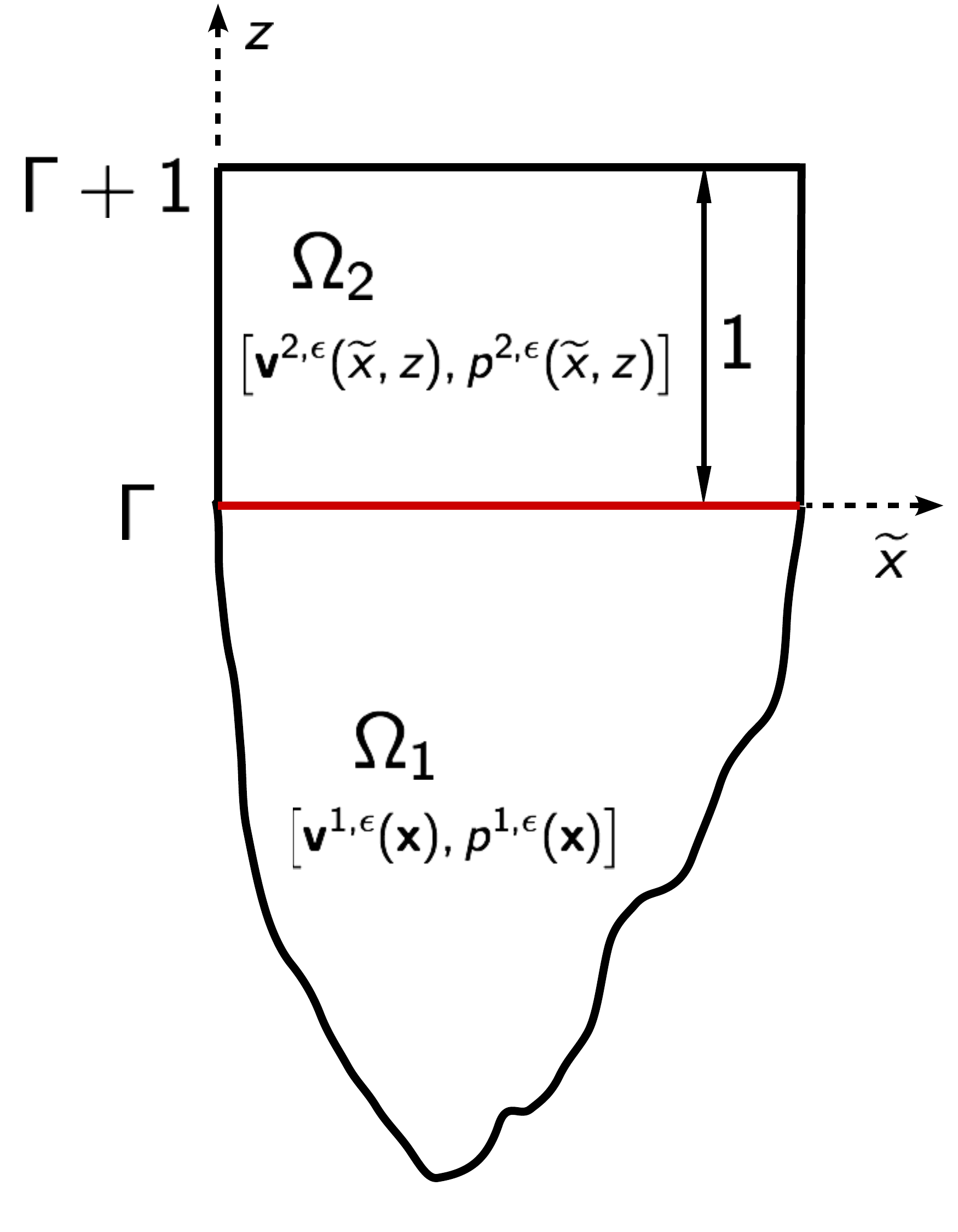} } }                
	\caption{The domain of reference for asymptotic analysis.}
	\label{Fig Domain Scaling}
\end{figure}
The solutions $\big\{\big[\,\veps, \peps\,\big]:\epsilon>0\big\}$ to
the Problem \eqref{scaled problem} (equivalently Problem \eqref{Def
Scaled Problem MIxed Formulation}), have different geometric domains
of definition and therefore no convergence statements can be
stated. In addition, the a-priori estimates given from the
well-posedness of the Problem \eqref{Def Scaled Problem MIxed
Formulation} depend on the geometry of the domain where the problem is
defined. Therefore, a domain of reference will be established; since
the only part that is changing is the thickness of the channel, this
suffices for the appropriate change of variable. Given
$\x=(\boldsymbol{\widetilde{x} } ,x_{\scriptscriptstyle N})\in\Omega_{2}^{\epsilon}$,
define $x_{\scriptscriptstyle N}=\epsilon\,z$, hence
$\frac{\partial}{\partial x_{\scriptscriptstyle
N}}=\frac{1}{\epsilon}\frac{\partial}{\partial z}$, see Figure
\ref{Fig Domain Scaling}. For any $\w\in\X_{2}^\epsilon$ we have the
following changes on the structure of the gradient and the divergence,
respectively,
\begin{subequations}\label{Eq Operators Change of Structure}
\begin{equation}\label{structure gradient}
\grad\,\w\,\big(\boldsymbol{\widetilde{x} } ,x_{\scriptscriptstyle N}\big)=
\left(\begin{array}{cc} [\,\grad_{\!T}\,\w_{\scriptscriptstyle T}] 
& \dfrac{1}{\epsilon}\,\partial_{z}\mathbf{w}_{\scriptscriptstyle T}\\[9pt]
\big(\,\grad_{\!T}\,\w_{\scriptscriptstyle N}\,\big)'
& \dfrac{1}{\epsilon}\,\partial_{z}\w_{\scriptscriptstyle N} \end{array}
\right)(\widetilde{x},z),
\end{equation}
\begin{equation}\label{structure divergence}
\div
\w\,(\boldsymbol{\widetilde{x} },x_{\scriptscriptstyle N})
=\Big(\grad_{\!T}\cdot\,\w_{\scriptscriptstyle T}
+\frac{1}{\epsilon}\,\partial_{z}\w_{\scriptscriptstyle N}\Big)(\widetilde{x},z) .
\end{equation}
\end{subequations}
Taking in consideration \eqref{structure gradient}, \eqref{structure
divergence} and combining it with \eqref{scaled problem} we obtain a
family of $\epsilon$-problems in a common domain of definition (see
Figure \ref{Fig Domain Original}) given by $\Omega\defining
\Omega_{1}\cup \Omega_{2}$, where $\Omega_{1},\Omega_{2}\subseteq
\R^{3}$ are bounded open sets, with $\Omega_{2}\defining \Gamma\times
(0,1)$ and $\Gamma=\partial
\Omega_{1}\cap\partial\Omega_{2}\subseteq\R^{2}$. Letting
$\Gamma+1\defining\big\{(\boldsymbol{\widetilde{x} }, 1):\boldsymbol{\widetilde{x} }\in
\Gamma\big\}$, the functional setting is now independent of $\epsilon$
and defined by
\begin{displaymath}
\X_{2}\defining
\big\{\v\in\H1bold(\Omega_{2}):\v=\0\;\text{on}\,\partial\Gamma\times(0,1),\,\
\v\cdot\n = 0\,\text{on}\,\Gamma+1\,\big\} ,
\end{displaymath}
\begin{equation}\label{Def Coupled Velocities Space}
\X\defining\big\{[\,\v^{1},\v^{2}\,]\in
\Hdiv(\Omega_{1})\times\X_{2}:\v^{1}\cdot\n=\v^{2}\cdot\n\;\text{on}\,\Gamma\big\} 
= \big\{ \v \in \Hdiv(\Omega): \v^{2} \in \X_2 \big\},
\end{equation}
\begin{displaymath}
\Y\defining L^{2}(\Omega) .
\end{displaymath}
Moreover, we have the following result.
\begin{proposition}\label{Th Scaled Weak Variational Formulation}
Under the change of variable $(\xthilde, x_{\scriptscriptstyle N})
\mapsto (\xthilde, x_{\scriptscriptstyle N}) \ind_{\Omega_{1}} +
(\xthilde, \epsilon \, z)\ind_{\Omega_{2}}$, the Problem \eqref{scaled
problem} is equivalent to the corresponding problem 
\begin{flushleft}
$[\v^{\epsilon},p^{\epsilon}]\in\X\times\Y:$
\end{flushleft}
\begin{subequations}\label{problem fixed geometry}
%
%
%
\begin{equation}\label{problem fixed geometry 1}
\begin{split}
\int_{\Omega_1}  \Q\, \vepsone \cdot \w^1 \,dx 
& - \int_{\Omega_1}\pepsone \, \grad\cdot\wone
\,dx 
-  \epsilon \int_{\Omega_2} \pepstwo \, \grad_{\!T} \cdot
\wtan \,d\widetilde{x} \,dz
- \int_{\Omega_2} \pepstwo \, \partial_{z} \wnorm \,d\widetilde{x} \,dz\\
& +\epsilon^{2}\int_{\Omega_{2}}\E
\grad_{\!T}\,\vtaneps:\grad_{\!T}\,\wtan\;d\widetilde{x}\,dz
\,+ \int_{\Omega_{2}}\E\,
\partial_{z}\vtaneps\cdot\,\partial_{z}\wtan\;d\widetilde{x}\,dz  \\
& +\epsilon^{2}\int_{\Omega_{2}}\E
\,\grad_{\!T}\vnormeps\cdot\grad_{\!T}\wnorm\;d\widetilde{x}\,dz
+\int_{\Omega_{2}}\E\, \partial_{z}\vnormeps\,\partial_{z}\wnorm\;d\widetilde{x}\,dz  \\
& +\alpha
\int_{\,\Gamma}\big(\,\vepsone\cdot\n\,\big)\,\big(\,\wone\,\cdot\n\,\big)\,d
S
+\epsilon^{2} \int_{\Gamma} \beta \, \sqrt{\Q} \,\vtaneps \cdot
\wtan
 \,d S 
= \epsilon\,\int_{\Omega_2} {\mathbf f^{2,\epsilon}} \cdot
\w^{2} \,d\widetilde{x} \,dz ,
\end{split}
\end{equation}
\begin{multline}\label{problem fixed geometry 2}
\int_{\Omega_1}\grad\cdot\vepsone
\varphi^1 \,dx  + \epsilon \int_{\Omega_2}  \grad_{\!T}\cdot\vtaneps
\,\varphi^2
 \,d\widetilde{x} \,dz 
 +  \int_{\Omega_2}  \partial_{z} \vnormeps \,\varphi^2
 \,d\widetilde{x} \,dz
= \int_{\Omega_1} h^{1,\,\epsilon} \, \varphi^1 \, dx ,
\\
\text{ for all } [\w,\Phi]\in\X\times\Y ,
\end{multline}
\end{subequations}
defined on the common domain of reference $\Omega$.
\end{proposition}
\begin{proof}
The proof follows from direct substitution together with the
identities \eqref{structure gradient} and \eqref{structure
divergence}.
\qed
\end{proof}
\begin{remark}\label{Rem Notation}
In order to avoid overloaded notation, from now on, we denote the
volume integrals by $\int_{\Omega_{1}} F = \int_{\Omega_{1}} F \, dx$
and $\int_{\Omega_{2}} F = \int_{\Omega_{2}} F \, d\xthilde\,dz$. The
explicit notation $\int_{\Omega_{2}} F \, d\xthilde\,dz$ will be used
only for those cases where specific calculations are needed. Both
notations will be clear from the context.
\end{remark}
%
%
%
%
\begin{proposition}\label{Th Strong Form}
The Problem \eqref{problem fixed geometry} is a weak
formulation of the strong form
\begin{subequations}\label{strong problem fixed geometry}
\begin{equation}\label{Darcy's law Strong problem}
\Q \,\vepsone +\grad \pepsone= \0\,,
\end{equation}
\begin{equation}\label{mass conservation Omega 1}
\grad\cdot\vepsone =
h^{1,\epsilon}\quad\text{in}\quad\Omega_{1}.
\end{equation}
\begin{equation}  \label{tangential momentum conservation modified}
\epsilon \,\grad_{\!T}\pepstwo-\epsilon^{2}\,\grad_{\!T}\cdot
\E\grad_{\!T}\,\vtaneps
-\,\partial_z\,\E
\,\partial_z\,\vtaneps
=\epsilon \,\f^{2,\,\epsilon}_{T}\,,
\end{equation}
\begin{equation}  \label{normal momentum conservation law}
\partial_z\, \pepstwo-\epsilon^{2}\,\grad_{\! T}\,\cdot\,\E
\grad_{\!T}\vnormeps-
\,\partial_z\,\E\,\partial_z \vnormeps = \epsilon\,
\f^{2,\,\epsilon}_{N}\,,
\end{equation}
\begin{equation}  \label{mass conservation Omega 2}
\epsilon\,\grad_{\! T}\cdot\vtaneps + \partial_z \vnormeps
=\,0\quad\text{in}\quad\Omega_2.
\end{equation}
\begin{equation}\label{strong normal interface condition}
\epsilon\,\E\,\partial_z\,\vnormeps -\pepstwo+\pepsone=\alpha\,
\vepsone\cdot\n\,,
\end{equation}
\begin{equation}  \label{scaled Beavers - Joseph condition}
\epsilon\,\E\,\frac{\partial\,\vepstwo_\tau}{\partial\,\n}
= \epsilon\,\E\,\partial_z\,\vtaneps
=\epsilon^{2}\beta\,\sqrt{\Q}\,\vtaneps\,,
\end{equation}
\begin{equation}  \label{scaled admissability constraint}
\vepsone\cdot\n=\vepstwo\cdot\n\quad\text{on}\quad\Gamma ,
\end{equation}
\begin{equation}  \label{boundary condition Omega 1}
\pepsone=0\quad\text{on}\quad\partial\Omega_1-\Gamma ,
\end{equation}
\begin{equation}  \label{no slip boundary condition Omega 2}
\vepstwo=\0\quad\text{on}\quad\partial\Gamma\times(0,1) ,
\end{equation}
\begin{equation}  \label{boundary conditio Omega 2}
\vepstwo\cdot\n = \vnormeps =
0\,,
\end{equation}
\begin{equation}  \label{strong tangential boundary condition}
\E\,\frac{\partial\,\vepstwo_\tau}{\partial\,\n} = \E\,\partial_z\,\vtaneps=\0\quad\text{on}\quad\Gamma+1 .
\end{equation}
\end{subequations}
\end{proposition}
\begin{sketch}
The strong Problem \eqref{strong problem fixed geometry} is obtained
using the standard procedure for recovering strong forms. First the
strong equations \eqref{Darcy's law Strong problem}, \eqref{mass
conservation Omega 1}, \eqref{tangential momentum conservation
modified}, \eqref{normal momentum conservation law} and \eqref{mass
conservation Omega 2} are recovered by testing the weak variational
Problem \eqref{problem fixed geometry} with compactly supported
functions. Next, the standard integration by parts with suitable test
functions recovers the boundary conditions \eqref{boundary condition
Omega 1}, \eqref{no slip boundary condition Omega 2}, \eqref{boundary
conditio Omega 2}, \eqref{strong tangential boundary condition} and
the interface conditions \eqref{strong normal interface condition},
\eqref{scaled Beavers - Joseph condition}, respectively. Finally, the
admissibility constraint \eqref{scaled admissability constraint} comes
from the modeling space $\X$ defined in \eqref{Def Coupled Velocities
Space}.
\qed
\end{sketch}
%
%
%
%
%
%
\section{Convergence Statements}  \label{sec-convergence}
%
%
%
%
We begin this section recalling a classical space. 
\begin{definition}\label{Def One Derivative Space}
Let $\Omega_{2} \defining \Gamma \times (0,1)$, define the Hilbert space
\begin{subequations}\label{Def Higher Order Normal Trace Space}
\begin{equation}\label{H partial z}
\Hpartial\defining
\big\{ u\in L^{2}(\Omega_{2}):
\partial_z\,u\in\,L^{2}(\Omega_{2})\big\} ,
\end{equation}
endowed with the inner product
\begin{equation}\label{inner product partial z}
\big\langle\,u, v\,\big\rangle_{ 
\Hpartial}
\defining \int_{\,\Omega_2}(u\,v + \partial_z  u \, \partial_z v\,)\,dx .
\end{equation}
\end{subequations}
\end{definition}
\begin{lemma}\label{Th Trace on One Derivative Space}
Let $\Hpartial$ be the space introduced in Definition \ref{Def One
Derivative Space}, then the trace map $u \mapsto u \big\vert_{\Gamma}$
from $\Hpartial$ to $L^{2}(\Gamma)$ is well-defined. Moreover, the
following Poincar\'e-type inequalities hold in this space
\begin{subequations}\label{Ineq Estimates on One Derivative Space}
\begin{equation}\label{Ineq Trace on One Derivative Space}
\big\Vert u\big\Vert_{0,\Gamma}\;
\leq 
\sqrt{2} \, \Big(\Vert u \Vert_{0, \Omega_{2}}
+ \big\Vert  \partial_{\,z}\,u\big\Vert_{0,\Omega_{2}} \Big),
\end{equation}
\begin{equation}\label{Ineq Conrol by Trace on One Derivative Space}
\big\Vert u\big\Vert_{0,\Omega_{2}}\;
\leq \sqrt{2} \, \Big(\big\Vert 
\partial_{\,z}\,u\big\Vert_{0,\Omega_{2}}
+\Vert u \Vert_{0,\Gamma} \Big),
\end{equation}
\end{subequations}
for all $u \in \Hpartial$.
\end{lemma}
\begin{proof}
The proof is a direct application of the fundamental theorem of
calculus on the smooth functions $C^{\infty}(\Omega_{2})$ which is a
dense subspace in $\Hpartial$.
\qed
\end{proof}
In order to derive convergence statements, it will be shown, accepting the next hypothesis, that the sequence of solutions is globally bounded.
\begin{hypothesis}\label{Hyp Bounds on the Forcing Terms}
In the following, it will be assumed that the sequences
$\{\f^{2,\epsilon}: \epsilon > 0\}\subseteq
\mathbf{L}^{2}(\Omega_{2})$ and $\{h^{1,\epsilon}: \epsilon >
0\}\subseteq L^{2}(\Omega_{1})$ are bounded, i.e., there exists $C>0$
such that
\begin{align}\label{Ineq Bounds on the Forcing Terms}
& \big\Vert \f^{2, \epsilon} \big\Vert_{0, \Omega_{2} }
\leq  C,&
& \big\Vert h^{1, \epsilon} \big\Vert_{0, \Omega_{1} } 
\leq C,&
& \text{for all } \, \epsilon > 0 . 
\end{align}
\end{hypothesis}
\begin{theorem}\label{Th A-priori Estimates of Velocity}
Let $[\veps, \peps]\in \X\times \Y$ be the solution to the Problem
\eqref{problem fixed geometry}. There exists a $K>0$ such that
\begin{equation}\label{general a priori estimate}
\begin{split}
\big\Vert\vepsone\big\Vert_{0,\Omega_{1}}^{2}
& +\big\Vert 
\,\grad_{\!T}\big(\,\epsilon\,\vtaneps\,\big)\big\Vert_{0,\Omega_{2}}^{2}
 +\big\Vert \partial_{z}\vtaneps\big\Vert_{0,\Omega_{2}}^{2}\\
& +\big\Vert \epsilon\,\grad_{\!T}\vnormeps
\big\Vert_{0,\Omega_{2}}^{2}
+\big\Vert \partial_{z}\vnormeps\big\Vert_{0,\Omega_{2}}^{2}
+\big\Vert\vnormeps \big\Vert_{0,\Gamma}^{2}
+\big\Vert\epsilon\,\vtaneps \big\Vert_{0,\Gamma}^{2}
 \leq K ,
 \qquad  \qquad \text{for all }\, \epsilon > 0.
 \end{split}
\end{equation}
\end{theorem}
\begin{proof}
Set $\w = \v^{\epsilon}$ in \eqref{problem fixed geometry 1} and
$\varphi = p^{\epsilon}$ in \eqref{problem fixed geometry 2}; add
them together to get
\begin{equation}\label{Eq Evaluation on the Diagonal}
\begin{split}
\int_{\Omega_1}  \Q\, \vepsone \cdot \vepsone 
& +\int_{\Omega_{2}}\E
\grad_{\!T}\big(\,\epsilon\,\vtaneps\,\big):
\grad_{\!T}\big(\,\epsilon\,\vtaneps\,\big) 
+\int_{\Omega_{2}}\E\,\partial_{z}\vtaneps\cdot
\partial_{z}\vtaneps 
\\
& +\epsilon^{2}\int_{\Omega_{2}}\E\,\grad_{\!T}\vnormeps\cdot
\grad_{\!T}\vnormeps 
+\int_{\Omega_{2}}\E\, \partial_{z}\vnormeps\,\partial_{z}\vnormeps 
\\
& +\alpha \int_{\,\Gamma} \big(\,\vepsone\cdot
\n\,\big)\,\big(\,\vepsone\cdot \n\,\big)\, d S
 + \int_{\Gamma}  \epsilon^{2}\,\beta \sqrt{\Q} \,\vtaneps \cdot
\vtaneps
 \,d S 
= \epsilon\, \int_{\Omega_2} {\mathbf f^{2,\epsilon}} \cdot
\vepstwo 
+\int_{\Omega_1} h^{1,\epsilon} 
\pepsone .
%
\end{split}
\end{equation}
The mixed terms were canceled out on the diagonal.  Applying the
Cauchy-Schwartz inequality to the right hand side and recalling the
Hypothesis \ref{Hyp Bounds on the Coefficients}, we get
\begin{multline}\label{a priori estimate}
\big\Vert\vepsone\big\Vert_{0,\Omega_{1}}^{2}
+\big\Vert 
\,\grad_{\!T}\big(\,\epsilon\,\vtaneps\,\big)\big\Vert_{0,\Omega_{2}}^{2}
+\big\Vert \partial_{z}\vtaneps\big\Vert_{0,\Omega_{2}}^{2}
+\big\Vert \epsilon\,\grad_{\!T}\vnormeps
\big\Vert_{0,\Omega_{2}}^{2}
+\big\Vert \partial_{z}\vnormeps\big\Vert_{0,\Omega_{2}}^{2}\\
+\big\Vert\vnormeps \big\Vert_{0,\Gamma}^{2}
+\big\Vert\epsilon\,\vtaneps \big\Vert_{0,\Gamma}^{2} 
\leq \frac{1}{k} \Big(\big\Vert {\mathbf
f^{2,\epsilon}}\big\Vert_{0,\Omega_{2}} \big\Vert
\,\epsilon\, \vepstwo \big\Vert_{0,\Omega_{2}}
+\int_{\Omega_1} h^{1,\,\epsilon} \, \pepsone \,\Big) .
\end{multline}
The summand involving an integral needs a special treatment in order
to attain the a-priori estimate.
\begin{equation}\label{estimate on the first integral}
\int_{\Omega_{1}} h^{1,\,\epsilon} \, \pepsone 
\leq  \big\Vert
\pepsone \big\Vert_{0,\Omega_{1}}\big\Vert
h^{1,\epsilon} \big\Vert_{0,\Omega_{1}} 
\leq C \, \big\Vert
\grad\,\pepsone \big\Vert_{0,\Omega_{1}}\big\Vert
h^{1,\epsilon} \big\Vert_{0,\Omega_{1}} 
=C \,\big\Vert \Q\,\vepsone \big\Vert_{0,\Omega_{1}}\big\Vert
h^{\,1,\,\epsilon} \big\Vert_{0,\Omega_{1}}
\leq \widetilde{C} \,\big\Vert \vepsone \big\Vert_{0,\Omega_{1}} .
\end{equation}
The second inequality holds due to Poincar\'e's inequality given that $\pepsone = 0$ on $\partial \Omega_{1} - \Gamma$, as stated in Equation \eqref{boundary condition Omega 1}. The equality holds due to \eqref{Darcy's law Strong problem}. The third inequality holds because the tensor $\Q$ and the family of sources $\{h^{1,\epsilon}:\epsilon>0\}\subset L^{2}(\Omega_{1})$ are bounded as stated in Hypothesis \ref{Hyp Bounds on the Coefficients} and \eqref{Ineq Bounds on the Forcing Terms}, Hypothesis \ref{Hyp Bounds on the Forcing Terms} respectively.
%
%
%
%
%
%
%
%
%
%
%
%
%
%
Next, we control the $L^2(\Omega_{2})$-norm of $\vepstwo$. Recalling that $\vepstwo\in \big[\Hpartial\big]^{N}$ then, a direct application of Estimate \eqref{Ineq Conrol by Trace on One Derivative Space} implies
\begin{equation}\label{poincare}
\big\Vert\vepstwo\big\Vert_{0,\Omega_{2}}\;\leq
\sqrt{2}\, \Big(\big\Vert 
\partial_{\,z}\,\vepstwo\big\Vert_{0,\Omega_{2}}
+ \Vert \vepstwo \Vert_{0,\Gamma} \Big).
\end{equation}
Combining \eqref{estimate on the first
integral}, \eqref{poincare} and the bound \eqref{Ineq Bounds on the Forcing Terms} from Hypothesis \ref{Hyp Bounds on the Forcing Terms} in \eqref{a priori estimate} we have
\begin{equation*}
\begin{split}
\big\Vert \vepsone \big\Vert_{0,\Omega_{1}}^{2}
+\big\Vert \,\grad_{\!T}\big(\,\epsilon\,\vtaneps\,\big)\big\Vert_{0,\Omega_{2}}^{2}
& +\big\Vert \partial_{z}\vtaneps\big\Vert_{0,\Omega_{2}}^{2} 
+\big\Vert \epsilon\,\grad_{\!T}\vnormeps
\big\Vert_{0,\Omega_{2}}^{2}
+\big\Vert \partial_{z}\vnormeps\big\Vert_{0,\Omega_{2}}^{2}
+\big\Vert\vnormeps \big\Vert_{0,\Gamma}^{2}
+\big\Vert\epsilon\,\vtaneps \big\Vert_{0,\Gamma}^{2}\\
& \leq C \Big[\big\Vert {\mathbf
f^{\,2,\,\epsilon}}\big\Vert_{0,\Omega_{2}} 
\sqrt{2} \, \big(\big\Vert 
\partial_{\,z}\,\big(\epsilon\,\vepstwo\big)\big\Vert_{0,\Omega_{2}}
+\big\Vert \big(\epsilon\,\vepstwo\big)\big\Vert_{0,\Gamma}\big)
+\widetilde{C}\big\Vert
\,\vepsone\big\Vert_{0,\Omega_{1}}\Big]\\
& \leq \widehat{C} 
\Big( \big\Vert 
\partial_{\,z}\,\vtaneps\big\Vert_{0,\Omega_{2}}
+ \big\Vert 
\partial_{\,z}\,\vepstwo\big\Vert_{0,\Omega_{2}}
+\big\Vert \epsilon\,\vtaneps\big\Vert_{0,\Gamma}\big. 
\big.
+\big\Vert \vnormeps\big\Vert_{0,\Gamma}+
\big\Vert
\,\vepsone\big\Vert_{0,\Omega_{1}}\Big) .
\end{split}
\end{equation*}
Using the equivalence of norms $\Vert \cdot \Vert_{\,1}\,,
\Vert \cdot \Vert_{\,2}$ for 5-D vectors yields
\begin{multline*} 
\big\Vert\vepsone\big\Vert_{0,\Omega_{1}}^{2}
+\big\Vert \,\grad_{\!T}\big(\,\epsilon\,\vtaneps\,\big)\big\Vert_{0,\Omega_{2}}^{2}
+\big\Vert \partial_{z}\vtaneps\big\Vert_{0,\Omega_{2}}^{2} 
+\big\Vert \epsilon\,\grad_{\!T}\vnormeps
\big\Vert_{0,\Omega_{2}}^{2}
+\big\Vert \partial_{z}\vnormeps\big\Vert_{0,\Omega_{2}}^{2}
+\big\Vert\vnormeps \big\Vert_{0,\Gamma}^{2}
+\big\Vert\epsilon\,\vtaneps \big\Vert_{0,\Gamma}^{2}\\
\leq C\,' \Big\{\big\Vert 
\partial_{\,z}\,\vtaneps\big\Vert^{2}_{\,0,\,\Omega_{2}}
+ \big\Vert 
\partial_{\,z}\,\vepstwo\big\Vert^{2}_{\,0,\,\Omega_{2}}
+\big\Vert \epsilon\,\vtaneps\big\Vert^{2}_{\,0,\,\Gamma}
%
+\big\Vert \vnormeps\big\Vert^{2}_{\,0,\,\Gamma}+
\big\Vert
\,\vepsone\big\Vert^{2}_{\,0,\,\Omega_{1}}\Big\}^{1/2}\\
\leq C\,\Big\{\big\Vert\vepsone\big\Vert_{0,\Omega_{1}}^{2}
+\big\Vert 
\,\grad_{\!T}\big(\,\epsilon\,\vtaneps\,\big)\big\Vert_{0,\Omega_{2}}^{2}
+\big\Vert \partial_{z}\vtaneps\big\Vert_{0,\Omega_{2}}^{2}\\
+\big\Vert \epsilon\,\grad_{\!T}\vnormeps\big\Vert_{0,\Omega_{2}}^{2}
+\big\Vert \partial_{z}\vnormeps\big\Vert_{0,\Omega_{2}}^{2}
+\big\Vert\vnormeps \big\Vert_{0,\Gamma}^{2}
+\big\Vert\epsilon\,\vtaneps \big\Vert_{0,\Gamma}^{2}
\Big\}^{1/2} .
\end{multline*}
The expression above implies the existence of a constant $K> 0$ satisfying the global Estimate \eqref{general a priori estimate}.
\qed
\end{proof}
In the next subsections we use weak convergence arguments to derive the functional setting of the limiting problem.  
%
%
%
%
\subsection{Weak Convergence of Velocity and Pressure}\label{Sec Weak Convergence os Subsequence}
%
%
%
%
We begin this part with a direct consequence of Theorem \ref{Th A-priori Estimates of Velocity}.
\begin{corollary}\label{Th Direct Weak Convergence of Velocities}
Let $[\veps, \peps]\in \X\times \Y$ be the solution to the Problem
\eqref{problem fixed geometry}. There exists a subsequence, still
denoted $\{\veps:\epsilon>0\}$ for which the following hold.
\begin{enumerate}[(i)]
\item There exist $\vone\in \Hdiv(\Omega_{1}) $ and $\vtan\in
\big[H^{1}(\Omega_{2})\big]^{N-1}$ such that
%
\begin{subequations}\label{Stmt Lower Order Velocities Weak Convergence}
\begin{equation}\label{velocity omega 1}
\vepsone\rightarrow \vone\quad\text{weakly in}\quad
\Hdiv(\Omega_1) .
\end{equation}
\begin{equation}\label{funcL2weak}
\epsilon\,\vtaneps\rightarrow\vtan
\quad\text{weakly in}\quad \big[H^{1}(\Omega_{2})\big]^{N-1} ,
\quad\text{strongly in}\quad \big[L^{2}(\Omega_{2})\big]^{N-1} ,
\end{equation}
%
%
\end{subequations}

\item There exist $\xi\in \Hpartial$ and $\eta\in \big[L^{2}(\Omega_{2})\big]^{N-1}$ such that 
\begin{subequations}\label{Stmt Higher Order Velocity Weak Convergence}
\begin{equation}\label{Eq epsilon partial tangential L2 weak}
\partial_{z}\vtaneps \rightarrow \eta \quad\text{weakly in}\quad 
\big[L^{2}(\Omega_{2})\big]^{N-1}\,, 
\quad\partial_{z} \big(\epsilon\,\vtaneps\big)
\rightarrow 0 \quad \text{strongly in}\quad 
\big[L^{2}(\Omega_{2})\big]^{N-1},
\end{equation}
\begin{equation}\label{convergence of normal velocity}
\vnormeps \rightarrow \xi \quad\text{weakly in }\, L^{2}(\Omega_{2})\,, 
\quad \big(\epsilon\,\vnormeps\big)
\rightarrow 0 \quad \text{strongly in} \; \Hpartial ,
\end{equation}
moreover, $\xi$ satisfies the interface and boundary conditions
\begin{align}\label{boundary conditions on normal velocity}
& \xi\vert_{\Gamma} = \vone\cdot \n\vert_{\Gamma} \, , &
& \xi\,(\widetilde{x},1) = 0 .
\end{align}
\end{subequations}

\item The limit function $\vtan$ satisfies that (see Figure \ref{Fig Limit Solution Schematics})
\begin{equation}\label{solutionconvergence}
\vtan=\vtan\,(\widetilde{x}) \, .
\end{equation}
\end{enumerate}
\end{corollary}
\begin{proof}
\begin{enumerate}[(i)]
\item Due to the global a-priori Estimate \eqref{general a priori
estimate} there must exist a weakly convergent subsequence and
$\vtan\in \big[H^{1}(\Omega_{2})\big]^{N-1}$ such that
\eqref{funcL2weak} holds; together with $\vone\in \Hdiv(\Omega_{1}) $
such that \eqref{velocity omega 1} holds only in the weak
$\mathbf{L}^{2}(\Omega_{1})$-sense. Because of the Hypothesis \ref{Hyp
Bounds on the Forcing Terms} and \eqref{mass conservation Omega 1},
the sequence $\{\div\vepsone:\epsilon>0\} \subset L^{2}(\Omega_{1})$
is bounded. Then, there must exist yet another subsequence, still
denoted the same, such that \eqref{velocity omega 1} holds in the weak
$\Hdiv(\Omega_{1})$-sense and the first part is complete.
\item For the higher order terms $\partial_{z}\vnormeps$,
$\partial_{z}\vtaneps$, in view of the estimate \eqref{general a
priori estimate}, there must exist $\eta \in
\big[L^{2}(\Omega_{2})\big]^{N-1}$ for which \eqref{Eq epsilon partial
tangential L2 weak} holds. Next, the estimate \eqref{general a priori
estimate} combined with \eqref{Ineq Conrol by Trace on One Derivative
Space} imply that $\big\{\,\vnormeps: \epsilon>0\,\big\}$ is a bounded
sequence in $\Hpartial$, so \eqref{convergence of normal velocity}
holds. Moreover, since the trace applications $\vnormeps \mapsto
\vnormeps\big\vert_{ \Gamma }$ and $\vnormeps \mapsto
\vnormeps\big\vert_{ \Gamma+1 }$ are continuous in $\Hpartial$ and
$\vnormeps (\widetilde{x}, 1) = 0$, the properties \eqref{boundary
conditions on normal velocity} follow. This concludes the second part.
\item The property \eqref{solutionconvergence}, is a direct
consequence of \eqref{Eq epsilon partial tangential L2 weak}.  Hence,
the proof is complete.
\qed
\end{enumerate}
\end{proof}%
\begin{lemma}\label{Th Convergence of Pressure One}
Let $\big[\veps, \peps \big]\in \X\times \Y$ be the solution of
\eqref{problem fixed geometry}. There exists a subsequence, still
denoted $\big\{\peps:\epsilon>0\big\}$ verifying the following.
\begin{enumerate}[(i)]
\item There exists $\pone\in H^{1}(\Omega_{1})$ such that
\begin{equation}\label{convergence of the pressure in Omega_1}
\pepsone \rightarrow \pone \quad \text{weakly in }
H^{1}(\Omega_{1}), \; \text{strongly in }\,
L^{2}(\Omega_{1}) .
\end{equation}
\item There exists $\ptwo\in L^{2}(\Omega_{2})$ such that
\begin{equation}\label{convergence of the pressure in Omega_2}
\pepstwo \rightarrow \ptwo\quad \text{weakly in }\,  L^{2}(\Omega_{2}) .
\end{equation}
\item The pressure $p = \big[ \pone, \ptwo \big]$ belongs to $L^{2}(\Omega)$.
\end{enumerate}
\end{lemma}
\begin{proof}
\begin{enumerate}[(i)]
\item Due to \eqref{Darcy's law Strong problem} and \eqref{a priori
estimate} it follows that
\begin{equation*}
\big\Vert \grad\,\pepsone\big\Vert_{0,\Omega_{1}}=
\big\Vert \sqrt{\Q}\,\vepsone\big\Vert_{0,\Omega_{1}}\leq C ,
\end{equation*}
with $C>0$ an adequate positive constant. From \eqref{boundary
condition Omega 1}, the Poincar\'e inequality implies there exists
a constant $\widetilde{C}>0$ satisfying
\begin{align}\label{Ineq H^1 boundedness of pressure one}
& \big\Vert \pepsone\big\Vert_{1,\Omega_{1}}
\leq \widetilde{C} \, \big\Vert \grad\,\pepsone\big\Vert_{0,\Omega_{1}}, &
& \text{for all }\, \epsilon > 0  .
\end{align}
Therefore, the sequence $\{\pepsone: \epsilon > 0\}$ is bounded in
$H^{1}(\Omega_{1})$ and the Statement \eqref{convergence of the
pressure in Omega_1} follows directly.

\item 
In order to show that the sequence $\{\pepstwo: \epsilon > 0 \}$ is
bounded in $L^{2}(\Omega_{2})$, take any $\phi\in
C_{0}^{\infty}(\Omega_{2})$ and define the auxiliary function
\begin{equation}\label{negative antiderivative}
\varsigma(\widetilde{x},z)\defining
\int_{z}^{1}\phi(\widetilde{x},t)\,dt .
\end{equation}
By construction it is clear that $\Vert \varsigma
\Vert_{\,1,\,\Omega_2}\leq C \Vert \phi \Vert_{0,\Omega_2}$. Since
$\varsigma \, \ind_{\Gamma} \in L^{2}(\Gamma)\subseteq
H^{-1/2}(\partial \Omega_{1})$, due to Lemma \ref{Th Surjectiveness
from Hdiv to H^1/2}, there must exist a function $\wone\in
\Hdiv\,(\Omega_{1})$ such that
$\wone\cdot\n=\w^{2}\cdot\n=\varsigma(\widetilde{x},0) =
\int_{0}^{1}\phi\,(\widetilde{x},t)\,dt$ on $\Gamma$, $\wone\cdot\n =
0$ on $\partial\,\Omega_{1} - \Gamma$ and $\Vert \wone
\Vert_{\,\Hdiv(\Omega_1)}\leq \Vert \varsigma \Vert_{0,\Gamma}\leq C
\Vert \phi \Vert_{0,\Omega_2}$. Hence, the function
$\w^{2}=\big(\0_{T},\varsigma(\widetilde{x},z)\big)$ is such that $\w
\defining [\w^{1},\w^{2}]\in\X$; testing \eqref{problem fixed geometry
1} with $\w$ yields
\begin{multline}\label{relation partial z pressure and normal velocity}
\int_{\Omega_1} \Q\,\vepsone\cdot\wone 
- \int_{\Omega_1} \pepsone\,
\grad\cdot\wone 
%
+\alpha\int_{\,\Gamma}\big(\vepsone\cdot\n\big)\big(\wone\cdot\n\,\big)dS
+\int_{\Omega_2}\,\pepstwo\,\phi 
\\
+\epsilon^{2}\int_{\Omega_2}\E\,\grad_{\,T}\,\vnormeps
\cdot\,\grad_{\,T}\,\varsigma 
-\int_{\Omega_2}\E\,\partial_z\,\vnormeps\,\phi 
=\epsilon\,\int_{\Omega_2}\f^{\,2,\,\epsilon}_{N}\,\varsigma . 
\end{multline}
Applying the Cauchy-Schwarz inequality to the integrals and reordering we get
\begin{equation*}
\begin{split}
\Big\vert\int_{\Omega_2}\,
\pepstwo\,\phi 
\Big\vert
\leq &
C_1\big\Vert \vepsone\big\Vert_{0,\Omega_{1}}\big\Vert \wone\big\Vert_{0,\Omega_{1}}
+ \big\Vert \pepsone\big\Vert_{0,\Omega_{1}}\big\Vert \grad\cdot\wone\big\Vert_{0,\Omega_{1}}
+ C_2\big\Vert \vepsone\cdot\n\big\Vert_{0,\Gamma} \big\Vert \varsigma\big\Vert_{0,\Gamma} 
\\
& 
+\epsilon\,C_3 \, \big\Vert \grad_{\,T}\,\big(\epsilon\,\vnormeps\big)\big\Vert_{0,\Omega_2}
\big\Vert \grad_{T}\,\varsigma(\widetilde{x},z)\big\Vert_{0,\Omega_2}
+C_4 \, \big\Vert \partial_z\,\vnormeps\big\Vert_{0,\Omega_2}
\big\Vert \phi\big\Vert_{0,\Omega_2}
+ \big\Vert \epsilon\,\f^{\,2,\,\epsilon}_{N}\big\Vert_{0,\Omega_2}
\big\Vert \varsigma\big\Vert_{0,\Omega_2} .
\end{split}
\end{equation*}
Notice that due to the construction, all the norms depending on
$\wone$ and $\varsigma$, with the exception of $\grad_{\!T}\varsigma$,
are controlled by the norm $\Vert \phi \Vert_{0,\Omega_2}$. Therefore,
the above expression can be reduced to
\begin{multline*}
\Big\vert\int_{\Omega_2}\,
\pepstwo\,\phi 
\Big\vert\leq
C\Big(\big\Vert \vepsone\big\Vert_{0,\Omega_{1}}
+\big\Vert \pepsone\big\Vert_{0,\Omega_{1}}
+\big\Vert \vepsone\cdot\n\big\Vert_{0,\Gamma}
+\big\Vert \partial_z\,\vnormeps\big\Vert_{0,\Omega_2}
+ \big\Vert \epsilon\,\f^{2,\epsilon}_{N}\big\Vert_{0,\Omega_2}\Big)
 \Vert \phi \Vert_{0,\Omega_2}
\\
+\epsilon\big\Vert \grad_{\! T}\,\big(\epsilon\,\vnormeps\big)\big\Vert_{0,\Omega_2}
\big\Vert \grad_{\! T}\,\varsigma(\widetilde{x},z)\big\Vert_{0,\Omega_2}
%
\leq\,\widetilde{C}\,\Big(\Vert \phi \Vert_{0,\Omega_2}
+\epsilon
\big\Vert \grad_{\! T}\,\varsigma(\widetilde{x},z)\big\Vert_{0,\Omega_2}\Big) .
\end{multline*}
The last inequality holds since all the summands in the parenthesis
are bounded due to the estimates \eqref{general a priori estimate},
\eqref{Ineq H^1 boundedness of pressure one} and the Hypothesis
\ref{Hyp Bounds on the Forcing Terms}. Taking upper limit when
$\epsilon\rightarrow 0$, in the previous expression we get
\begin{equation}\label{bound on epsilon pressure in
L2}
\limsup_{\,\epsilon\,\downarrow\,0}\Big\vert\int_{\Omega_2}\,
\pepstwo\,\phi 
\Big\vert\leq \widetilde{C} \, \big\Vert \phi\big\Vert_{0,\Omega_2} .
\end{equation}
Since the above holds for any $\phi\in C_0^{\infty}(\Omega_{2})$, we
conclude that the sequence $\big\{\,\pepstwo:\epsilon>0\,\big\}\subset
L^{2}(\Omega_{2})$ is bounded and consequently \eqref{convergence of
the pressure in Omega_2} follows.

\item From the previous part it is clear that the sequence
$\big\{[\pepsone, \pepstwo]: \epsilon > 0\big\}$ is bounded in
$L^{2}(\Omega)$, so the proof is complete.
\qed
\end{enumerate}
\end{proof}
Finally, we identify the dependence of $\ptwo$ and $\xi$.
\begin{theorem}\label{Th Dependence of xi and Pressure 2}
Let $\xi$, $\ptwo$ be the higher order limiting term in Corollary
\ref{Th Direct Weak Convergence of Velocities} (ii) and the limit
pressure in $\Omega_{2}$ in Lemma \ref{Th Convergence of Pressure One}
(ii), respectively. Then (see Figure \ref{Fig Limit Solution Schematics}) we have
\begin{subequations}\label{Eq Dependence of xi and Pressure 2}
\begin{equation}  \label{partial xi dependence}
\partial_z\,\xi = \partial_z\,\xi\,(\widetilde{x}) ,
\end{equation}
\begin{equation}  \label{pressure two dependence}
\ptwo = \ptwo (\widetilde{x}) .
\end{equation}
\end{subequations}
\end{theorem}
\begin{proof}
Testing \eqref{problem fixed geometry 2} with $\varphi =
[\,0,\,\varphi^2\,]\in \Y$ and letting $\epsilon \rightarrow 0$
together with \eqref{funcL2weak} and \eqref{convergence of normal
velocity}, we get
\begin{equation*}  
\int_{\Omega_2}\,\grad_{\!T}\cdot
\vtan\,\varphi^2 
+ \int_{\Omega_2}\,\partial_z\, \xi\,\varphi^2 
= 0 ,
\end{equation*}
for all $\varphi^{2} \in L^{2}(\Omega_{2})$, consequently 
\begin{equation*}  
\grad_{\!T}\cdot \vtan+ \,\partial_z\, \xi\,=  0 ,
\end{equation*}
Now, due to the dependence of $\vtan$ from Corollary \ref{Th Direct
Weak Convergence of Velocities} (iii) the Identity \eqref{partial xi
dependence} follows.

For the Identity \eqref{pressure two dependence}, take the limit as
$\epsilon\downarrow 0$ in \eqref{relation partial z pressure and
normal velocity}; since the sequence
$\big\{[\pepsone,\pepstwo]:\epsilon>0\big\}$ is weakly convergent as
seen in Lemma \ref{Th Convergence of Pressure One}, this yields
\begin{equation*}
\int_{\Omega_1}\Q\,\vone\cdot\wone 
- \int_{\Omega_1} \pone\,\grad\cdot\wone 
+\alpha\int_{\,\Gamma}\xi\big(\wone\cdot\n\,\big)dS 
+\int_{\Omega_2}\,
\ptwo\,\phi 
-\int_{\Omega_2}\E\,\partial_z\,\xi\,\phi 
=0 .
\end{equation*}
Integrating by parts the second summand and using \eqref{Darcy's law Strong problem} we get
\begin{equation*}
-\int_{\Gamma} \pone\,\big(\wone\cdot\n\big) dS
+\alpha\int_{\,\Gamma}\xi\big(\wone\cdot\n\,\big)dS 
+\int_{\Omega_2}\,
\ptwo\,\phi 
-\int_{\Omega_2}\E\,\partial_z\,\xi\,\phi 
=0 .
\end{equation*}
Recalling that $\wone\cdot\n \,\vert_{\,\Gamma}\,=
\int_{\,0}^{1}\phi\,(\widetilde{x},z)\,dz$, we see the above
expression transforms into
\begin{multline*}
-\int_{\Gamma} \pone
\mid_{\,\Gamma}\,
\Big(\int_{0}^{1}\phi(\widetilde{x},t)\,dt\Big)\,d\widetilde{x}
+\alpha\int_{\Gamma}\xi\mid_{\,\Gamma}\,
\Big(\int_{0}^{1}\phi(\widetilde{x},t)\,dt\Big)\,d\widetilde{x}\\
+\int_{\Omega_2}\,
\ptwo\,\phi(\widetilde{x},z)\,d\widetilde{x}\,dz -
\int_{\Omega_2}\E\,\partial_z\,\xi\,\phi(\widetilde{x},z)\,d\widetilde{x}\,dz=0 .
\end{multline*}
The above holds for any $\phi\in C_0^{\infty}(\Omega_{2})$ and
$\xi\,\vert_{\Gamma},\,\pone\,\vert_{\Gamma}$ can be embedded in
$\Omega_{2}$ with the extension, constant with respect to $z$, to the
whole domain, so we conclude that
\begin{equation*} 
-\pone\,\vert_{\,\Gamma}+\alpha \,\xi \,\vert_{\,\Gamma}+\ptwo-\E\,\partial_z\,\xi=0
\; \;\text{in }\, L^{2}(\Omega_{2}) .
\end{equation*}
Together with \eqref{boundary conditions on normal velocity} this
shows
\begin{equation}  \label{second identification of xi}
\ptwo 
= \E\,\partial_z\,\xi 
- \alpha \,\vone\cdot\n \vert_{\Gamma}
+ \pone\,\vert_{\Gamma}  \; \;\text{in }\, L^{2}(\Omega_{2}) \, ,
\end{equation}
and then with \eqref{partial xi dependence}
we obtain \eqref{pressure two dependence}.
\qed
\end{proof}

We close this section with an equivalent form for \eqref{problem fixed
geometry} which will be useful in characterizing the limiting problem.
\begin{proposition}\label{Th rearanged problem}
The problem \eqref{problem fixed geometry} is equivalent to 
\begin{flushleft}
$ \big[\,\v^{\epsilon},\peps\,\big] \in \X\times\Y:$
\end{flushleft}
\vspace{-0.4cm}
\begin{subequations}\label{rearanged problem}
\begin{multline}\label{rearanged problem 1}
\int_{\Omega_1}  \Q\, \vepsone \cdot \w^1 
-
\int_{\Omega_1}\pepsone \, \grad\cdot\wone 
 -   \int_{\Omega_2} \pepstwo \, \grad_{\!T} \cdot\wtan 
- \int_{\Omega_2} \pepstwo \, \partial_{z} \wnorm 
\\
+\,\int_{\Omega_{2}}\E\grad_{\!T}\,\big(\epsilon\,\vtaneps\big):\grad_{\!T}\,\wtan 
+ \frac{1}{\epsilon^{2}}\int_{\Omega_{2}}\E\,
\partial_{z}\big(\epsilon\,\vtaneps\big)\cdot\,\partial_{z}\wtan 
\\
+ \epsilon\int_{\Omega_{2}}\E\,\grad_{\!T}\big(\epsilon\,\vnormeps\big)\cdot\grad_{\!T}\wnorm 
+\int_{\Omega_{2}}\E\, \partial_{z}\vnormeps\,\partial_{z}\wnorm 
\\
 +\alpha
\int_{\,\Gamma}\big(\,\vepsone\cdot\n\,\big)\,\big(\,\wone\,\cdot\n\,\big)\,dS
+ \int_{\Gamma} \beta \, \sqrt{\Q} \,\big(\epsilon\,\vtaneps\big) \cdot
\wtan \,d S 
= \epsilon\,\int_{\Omega_2} {\mathbf f^{\,2,\,\epsilon}} \cdot\wtwo ,
\end{multline}
\begin{equation}  \label{rearanged problem 2}
\int_{\Omega_1}\grad\cdot\vepsone\varphi^1 
+ \epsilon \int_{\Omega_2}  \grad_{\!T}\cdot\vtaneps\,\varphi^2
+  \int_{\Omega_2}  \partial_{z} \vnormeps \,\varphi^2 
= \int_{\Omega_1} h^{1,\,\epsilon} \, \varphi^1 ,
\end{equation}
\begin{flushright}
for all $[\w,\varphi]\in\X\times\Y$.
\end{flushright}
\end{subequations}
\end{proposition}
\begin{proof}
It is enough to observe that in the quantifier $\w =
[\,\wone,\wtwo]\in \X$, the tangential and normal components of
$\wtwo$ are decoupled. Therefore, the satisfaction of the Statement
\eqref{problem fixed geometry 1} for every $\big[\wone, (\wtan,
\wnorm)\big]\in \X$ or for every $\big[\wone, (\epsilon^{-1}\,\wtan,
\wnorm)\big]\in \X$ are equivalent logical statements; this proves the
result.
\qed
\end{proof}
%
%
\section{The Limiting Problem}  \label{sec-limiting}
%
In order to characterize the limiting problem, we introduce 
appropriate spaces. The limiting pressure space is given by
\begin{equation}\label{Def Limiting Pressure Space}
\Y^0\defining
\big\{(\varphi^{1},\varphi^{2})\in
\Y:\varphi^{2}=\varphi^{2}(\widetilde{x})\big\}.
\end{equation}
We shall exploit below the equivalence $\Y^0 \cong L^2(\Omega_1)
\times L^2(\Gamma)$.
The construction of the velocities limiting space is more
sophisticated. First define
\begin{subequations}\label{Def space of the limit velocity omega 2}
\begin{multline}\label{space of the limit velocity omega 2}
\X_2^0\defining \Big\{\,\wtwo=[\,\wtan, \wnorm\,]:\wtan\in
\big(\,H^{1}(\Omega_2)\,\big)^{N-1},\,\wtan = \wtan (\widetilde{x}),\, \\
\wtan = \0\,\text{on}\,\partial \Gamma,\,
\wnorm\in \Hpartial\,,\;\partial_z \wnorm = \partial_z\,\wnorm
(\widetilde{x}\,)\,,
 \,\wnorm(\widetilde{x},1) = 0\,\Big\} 
\end{multline}
endowed with its natural norm
\begin{equation}\label{norm space of the limit velocity omega 2}
\Vert \wtwo \Vert_{\,\X_2^0} = 
\Big( \Vert \wtan \Vert_{1,\,\Omega_2}^{2}
+\Vert \wnorm \Vert_{\,\Hpartial}^{2}\,\Big)^{1/2} .
\end{equation}
\end{subequations}
Next we introduce a subspace of $\X$ fitting the limiting process
together with its closure,
\begin{subequations}\label{Def Limiting Velocity Spaces}
\begin{equation}\label{Def Limiting Subspace in X}
\W\defining\big\{ \big(\w^{1},\wtwo\big)\in \X:\wtan =
\wtan(\widetilde{x}),\,\partial_z\,\wnorm =\partial_z\,\wnorm
(\widetilde{x})\,\big\} ,
\end{equation}
\begin{equation}\label{space of the limit velocity}
\X^0\defining\Big\{\,(\w^{1},\w^{2})\in
\Hdiv(\Omega_1)\times\X_2^0:\wone\cdot\n = \wnorm =
\wtwo\cdot\n\;\text{on}\,\Gamma\,\Big\} .
\end{equation}
\end{subequations}
Clearly $\W\subseteq \X^0 \cap \X$; before presenting the limiting
problem, we verify the density.
\begin{lemma}\label{Th Density Result}
The subspace $\W\subseteq \X$ is dense in $\X^0$.
\end{lemma}
\begin{proof} 
Consider an element $\w = (\wone, \wtwo)\in \X^0$, then $\wtwo =
(\wtan, \wnorm)\in \X_2^0$, where $\wnorm \in \Hpartial$ is completely
defined by its trace on the interface $\Gamma$. Given $\epsilon>0$
take $\varpi \in H^{1}_0(\Gamma)$ such that $\Vert \varpi -
\wnorm\vert_{\,\Gamma} \Vert_{L^{2}(\Gamma)}\leq \epsilon$. Now extend
the function to the whole domain by $\varrho(\widetilde{x},
z)\defining\varpi(\widetilde{x})(1-z)$, then $\Vert \varrho - \wnorm
\Vert_{\Hpartial}\leq \epsilon$. The function $(\wtan, \varrho)$
clearly belongs to $\W$. From the construction of $\varrho$ we know
that $\Vert \varrho\,\vert_{\,\Gamma} - \wnorm\,\vert_{\,\Gamma}
\Vert_{0,\Gamma}=\Vert \varpi - \wnorm\,\vert_{\,\Gamma}
\Vert_{0,\Gamma}\leq \epsilon$. Define $g = \varrho\,\vert_{\,\Gamma}
- \wnorm\,\vert_{\,\Gamma}\in L^{2}(\Gamma)$, due to Lemma \ref{Eq
Normal Trace Definition} there exists $\u\in \Hdiv(\Omega_{1})$ such
that $\u \cdot\n = g$ on $\Gamma$, $\u\cdot \n = 0$ on $\partial
\Omega_{1}-\Gamma$ and $\Vert \u \Vert_{\,\Hdiv (\Omega_1)}\leq C_1
\Vert g \Vert_{0,\Gamma}$ with $C_{1}$ depending only on
$\Omega_{1}$. Then, the function $\wone + \u$ is such that
$(\wone+\u)\cdot\n = \wone\cdot\n+\varpi -\wnorm = \varpi$ and $\Vert
\wone+\u - \wone \Vert_{\,\Hdiv(\Omega_1)} = \Vert\u
\Vert_{\,\Hdiv(\Omega_1)}\leq C_1 \Vert g \Vert_{0,\Gamma}\leq
C_1\,\epsilon$. Moreover, we notice that the function $(\wone+\u,
[\,\wtan, \varrho\,])$ belongs to $\W$, and due to the previous
observations we have
\begin{equation*}
\big\Vert \w - (\wone+\u\,, [\,\wtan, \varrho\,])  \big\Vert_{\X^0} =
\big\Vert (\wone,\wtwo) - (\wone+\u\,, [\,\wtan, \varrho\,])
 \big\Vert_{\X^0}\leq \sqrt{C_1 + 1 } \; \epsilon .
\end{equation*}
Since the constants depend only on the domains
$\Omega_{1}$ and $\Omega_{2}$, it follows that $\W$ is dense in
$\X^0$.
\qed
\end{proof}

Now we are ready to give the variational formulation of the limiting
problem.
\begin{theorem}\label{Th Formulation of the Limiting Problem}
Let $[\v, p]$, with $\vtwo = \big[\vtan, \xi\big]$, be the weak limits
obtained in Corollary \ref{Th Direct Weak Convergence of Velocities} and
Lemma \ref{Th Convergence of Pressure One}. Then $[\v, p]$ satisfies
the variational statement
\begin{subequations}\label{formulation limit problem}
\begin{flushleft}
$[\,\v,p\,]\in \X^0\times \Y^0:$
\end{flushleft}
\vspace{-0.4cm}
\begin{multline}  \label{formulation limit problem 1}
\int_{\Omega_1} \Q\, \vone \cdot \wone  
- \int_{\Omega_1} \pone \, \grad\cdot \wone
- \int_{\Omega_2} \ptwo \, \grad\cdot \big[\wtan, \wnorm\big] 
\\
+\int_{\Omega_{2}}\E\,\grad_{\!T}\,\vtan:\grad_{\!T}\,\wtan
+\int_{\,\Omega_2}\E\,\big(\,\partial_z\,\xi\,\big)\,\big(\,\partial_z
\wnorm\,\big) 
\\
+\alpha\,\int_{\Gamma}
\big(\,\vone\cdot\n\,\big)\,\big(\,\wone\cdot\n\,\big)\,d S
+ \int_{\Gamma} \beta \,  \sqrt{\Q} \,\,\vtan\, \cdot \wtan\,d S 
= \int_{\Omega_2} {\mathbf f}^{2}_{T}\cdot \wtan , 
\end{multline}
\begin{equation}  \label{formulation limit problem 2}
\int_{\Omega_1}\grad\cdot\vone \varphi^1 
+ \int_{\Omega_2} \grad\cdot[\,\vtan,\,\xi\,] \; \varphi^2 
= \int_{\Omega_1} h^{1} \, \varphi^1 , 
\end{equation}
\begin{flushright}
for all $[\,\w,\varphi \,]\in \X^0\times \Y^0.$
\end{flushright}
\end{subequations}
Moreover, the mixed variational formulation of the problem above is
given by
\begin{equation} \label{limit problem mixed formulation}
\begin{split}
[\,\v,p\,]\in \X^0\times\Y^0:
A\,\v - B '\,p&=\mathbf{f} , \\
B\,\v & = h ,
\end{split}
\end{equation}
%
where the forms $A: \X^0 \rightarrow (\X^0)'$ and $B: \X^0 \rightarrow
(\Y^0)'$ are defined by
\begin{subequations}\label{Def Limit Mixed Operators}
\begin{equation}\label{limit operator A}
A \defining
\left(\begin{array}{cc}\, \Q  +\,\gamma_{\, \n}
'\,\alpha\,\gamma_{\, \n} 
& \0 \\
\0 
& \big[\,\,\gamma_{T}'\,\beta\,\sqrt{\Q}\,\gamma_{T}
+\,(\grad_{\,T})'\,\E\,\grad_{\,T},(\partial_z)'\,\E\,\partial_z\,\big]\end{array}
\right) ,
\end{equation}
\begin{equation}\label{limit operator B }
B \defining
\left(\begin{array}{cc} \,\grad\cdot & 0 \\
0 &  \grad\cdot \end{array} \right)
=
\left(\begin{array}{cc} \,\Div & 0 \\
0 &  \Div \end{array} \right) .
\end{equation}
\end{subequations}
\end{theorem}
\begin{proof}
First, test the Problem \eqref{rearanged problem} with a function of
the form $[\w,\varphi ]\in \W\times\Y^0$. This gives
\begin{subequations} 
\begin{multline*} 
\int_{\Omega_1}  \Q\, \vepsone \cdot \wone 
-\int_{\Omega_1}\pepsone \, \grad\cdot\wone 
-  \int_{\Omega_2} \pepstwo \, \grad_{\!T} \cdot \wtan
-  \int_{\Omega_2} \pepstwo \,\partial_z\,\wnorm 
\\
+\int_{\Omega_{2}}\E\,\grad_{\!T}\big(\epsilon\,\vtaneps\big):\grad_{\!T}\,\wtan 
+\epsilon\int_{\,\Omega_2}
\E\,\grad_{\!T}\,\big(\epsilon\,\vnormeps\big)\cdot\grad_{\!T}\,\vnorm 
+\int_{\,\Omega_2}\E\,\partial_z\,\vnormeps\, \partial_z \wnorm 
\\
+ \alpha\int_{\Gamma}\big(\,\vepsone\cdot\n\,\big)\,\big(\,\wone\cdot\n\,\big)\,d S
+ \int_{\Gamma} \beta \, \sqrt{\Q} \,\big(\epsilon\,\vtaneps\big)\cdot \wtan\,d S = 
 \int_{\Omega_2} {\mathbf f}^{2,\epsilon}_{T}\,\cdot\, \wtan 
+\epsilon\int_{\Omega_2}  \f_{N}^{2,\epsilon} \cdot \wnorm , 
\end{multline*}
\begin{equation*}
\int_{\Omega_1}\grad\cdot\vepsone\, \varphi^1 
+  \int_{\Omega_2} \grad_{\!T}\cdot\big(\,\epsilon\,\vtaneps\,\big) \,\varphi^2
\\
+ \int_{\Omega_2}  \partial_{z} \vnormeps \,\varphi^2 
= \int_{\Omega_1} h^{1,\,\epsilon} \, \varphi^1 , 
\end{equation*}
\end{subequations}
and then letting $\epsilon\, \downarrow \,0$ yields
\begin{subequations}\label{limit problem on W}
\begin{multline}  \label{limit problem on W 1}
\int_{\Omega_1}  \Q\, \vone \cdot \wone 
- \int_{\Omega_1}\pone \,\grad\cdot\wone 
-  \int_{\Omega_2} \ptwo \, \grad_{\!T} \cdot \wtan 
-  \int_{\Omega_2} \ptwo \,\partial_z\,\wnorm 
\\
+\int_{\Omega_{2}}\E\,\grad_{\!T}\,\vtan:\grad_{\!T}\,\wtan 
+\int_{\,\Omega_2}\E\,\partial_z\,\xi\, \partial_z\wnorm 
\\
+ \alpha\,\int_{\Gamma}
\big(\,\vone\cdot\n\,\big)\,\big(\,\wone\cdot\n\,\big)\,d S
+ \int_{\Gamma} \beta \, \sqrt{\Q} \,\vtan \cdot \wtan
 \,d S 
= \int_{\Omega_2} {\mathbf f}^{2}_{\,T}
\,\cdot\, \wtan , 
\end{multline}
\begin{equation}  \label{limit problem on W 2}
\int_{\Omega_1}\grad\cdot\vone \varphi^1 
+ \int_{\Omega_2} \grad_{\!T}\cdot\big(\,\vtan\,\big)\,\varphi^2 
\\
+ \int_{\Omega_2}  \partial_{z} \xi \,\varphi^2 
= \int_{\Omega_1} h^{1} \, \varphi^1 . 
\end{equation}
\end{subequations}
Since the variational statements above hold for all $[\w,\varphi
]\in \W\times\Y^0$ and the bilinear forms are continuous with respect
to the space $\X^0\times \Y^0$, we can extend them by density to all
test functions $[\w,\varphi ]\in \X^0\times\Y^0$; these yield 
\eqref{formulation limit problem}. Finally, the mixed variational
characterization \eqref{limit problem mixed formulation} follows
immediately from the definition of the bilinear forms $A$ and $B$
given in \eqref{limit operator A} and \eqref{limit operator B },
respectively.
\qed
\end{proof}

The existence of a solution of Problem \eqref{limit problem mixed
formulation} follows from that of the limits above.  For an
independent proof of the well-posedness of Problem \eqref{limit
problem mixed formulation} we prepare the following intermediate
results.
\begin{lemma} \label{Th coercivity of A in the limit}
The operator $A$ is $\X^0$-coercive over $\X^0\cap \ker (B)$.
\end{lemma}
\begin{proof}
The form
$
\displaystyle A\v\,(\v)+\int_{\Omega_1}\big(\grad\cdot\v\big)^{2}$ is $\X^0$-coercive, and $\grad \cdot \v\,\vert_{\,\Omega_1}=0$ whenever $\v \in \ker(B)$.
\qed
\end{proof}
\begin{lemma}\label{Th B closed range in the limit}
The operator $B$ has closed range.
\end{lemma}
\begin{proof} 
Fix $\varphi = [\,\varphi^1,\,\varphi^{2}\,]\in \Y^0$. With
$\varphi^{2} = \varphi^{2} (\widetilde{x}\,)\in L^2(\Gamma)$, solve
the auxiliary problem
\begin{align}\label{Pblm Auxiliary Problem}
& -\div  \grad \phi = \varphi^{1}\quad\text{in }\, \Omega_{1}, &
& \grad \phi \cdot \n =  \varphi^{2} \quad \text{on }\, \Gamma, &
& \phi = 0 \quad \text{on }\, \partial \Omega_{1} - \Gamma .
\end{align}
Then $\uone = - \grad \phi$ satisfies $\div \uone = \varphi^{1}$, $\uone\cdot \n = - \varphi^{2}$ and 
\begin{equation}\label{Ineq Estimate on the Auxiliary function one}
\Vert \uone \Vert_{\Hdiv(\Omega_{1})} \leq C_{1}\, \big(\Vert \varphi^{1} \Vert_{0,\Omega_{1}}^{2} + \Vert \varphi^{2} \Vert_{0, \Omega_{2}}^{2} \big)^{1/2} ,
\end{equation}
because $\Vert \varphi^{2}\Vert_{L^{2}(\Gamma)} = \Vert
\varphi^{2}\Vert_{L^{2}(\Omega_{2})} $. Next, define $\unorm(\xthilde,
z) \defining - \varphi^{2}(\xthilde)(1 - z)$. The function
$\u =\big(\u^1, [\,\0_{T}, \unorm\,]\big)$
belongs to the space $\X^0$ (see Figure \ref{Fig Limit Solution
Schematics}), and
\begin{equation}\label{Ineq Estimate on the whole Auxiliary function}
\Vert \u \Vert_{\,\X^0}\leq C \big(
\Vert \u^1 \Vert_{\Hdiv(\Omega_{1})}^{2}
+\Vert \unorm\Vert_{\Hpartial }^{2}\big)^{1/2}\leq
\widetilde{C}
\big(\Vert \varphi^1 \Vert_{0,\Omega_{1}}^{2}+\Vert \varphi^{2} \Vert_{0,\Omega_2}^{2}\big)^{1/2} .
\end{equation}
Here $\widetilde{C}$ depends on the domains $\Omega_1, \Omega_{2}$ as well as 
the equivalence of norms for 2-D vectors, but it is independent
of $\varphi \in \Y^0$. Moreover, notice that $\div [\0_{T}, \unorm] = \varphi^{2}$. Hence, we have the inequalities
\begin{multline}\label{Ineq Chain of Inequalities for Closed Range}
\sup _{\w\in\X^0} \frac{\int_{\Omega}\varphi \,\grad\cdot\w\,dx}{\Vert \w \Vert_{\,\X^0}}
\geq 
\frac{\int_{\Omega_1}\varphi^1\, \grad \cdot \u^1\,dx
+\int_{\Omega_2}\varphi^{2} \,\div [\0_{T}, \unorm] \,d\widetilde{x}\,dz}{\Vert \u \Vert_{\,\X^0}}
\\
\geq \frac{1}{\widetilde{C}}\,
\frac{\Vert \varphi^1 \Vert_{0,\Omega_{1}}^{2}+\Vert \varphi^{2} \Vert_{0,\Omega_2}^{2}}
{(\Vert \varphi^1 \Vert_{0,\Omega_{1}}^{2}+\Vert \varphi^{2} \Vert_{0,\Omega_2}^{2})^{1/2}} 
=
\frac{1}{\widetilde{C}}\,\big(\Vert \varphi^1 \Vert_{0,\Omega_{1}}^{2}
+\Vert \varphi^{2} \Vert_{0,\Omega_2}^{2}\big)^{1/2}
\\
= \frac{1}{\widetilde{C}} \Vert \varphi  \Vert_{0,\Omega}\,,\quad
\forall\,\varphi \in \Y^0 .
\end{multline}
\qed
\end{proof}
\begin{theorem}\label{Th well-posedness of the limiting problem}
The Problem \eqref{limit problem
mixed formulation} is well-posed.
\end{theorem}
\begin{proof}
Due to Lemmas \ref{Th coercivity of A in the limit} and \ref{Th B
closed range in the limit} above, the operators $A$ and $B$ satisfy
the hypotheses of Theorem \ref{Th well posedeness mixed formulation
classic} and the result follows.
\qed
\end{proof}
%
%
%
\begin{remark}\label{Rem Regularity Demmand Discussion}
Note that the proof of Lemma \ref{Th B closed range in the limit} for
the limit problem \eqref{limit problem mixed formulation} is
substantially different from the corresponding Lemma \ref{Th B closed
range epsilon} for the $\epsilon$-problem \eqref{Def Scaled Problem
MIxed Formulation}. This is due to the respective spaces $\X^0$ and
$\X$; the condition $\wone\cdot \n = \wtwo\cdot \n$ on $\Gamma$ is
significantly different in terms of regularity, from one case to the
other. Specifically, in the case of the limit problem $\wone\cdot
\n \big\vert_{\Gamma}\in L^{2}(\Gamma)$, while in the
$\epsilon$-problem $\wone\cdot \n \big\vert_{\Gamma}\in
H^{1/2}(\Gamma)$. The demands of normal trace regularity on $\Gamma$
are weakened in the limit as a consequence of the upscaling process.
\end{remark}
%
%
%
\begin{corollary}\label{Th Weak convergence of the whole sequence}
Let $\big\{\big[\veps, \peps \big]: \epsilon > 0 \big\}\subseteq
\X\times \Y$ be the sequence of solutions to the family of problems
\eqref{problem fixed geometry}, then the whole sequence converges
weakly to $[\v, p]\in \X^0\times \Y^0$, the solution of Problem
\eqref{formulation limit problem}.
\end{corollary}
\begin{proof}
Due to the well-posedness shown in Theorem \ref{Th well-posedness of
the limiting Problem}, the solution $[\v, p]\in \X^0\times \Y^0$ of
problem \eqref{limit problem mixed formulation} is unique. On the
other hand, all the reasoning from Section \ref{Sec Weak Convergence
os Subsequence} on, is applicable to any subsequence of
$\big\{\big[\veps, \peps \big]: \epsilon > 0 \big\}$; which yields a
further subsequence weakly convergent to $[\v, p]$. Hence, the result
follows.
\qed
\end{proof}
\begin{figure}[h] 
	\centering
	\begin{subfigure}[Limit Solutions in the Domain of Reference. ]
		{\resizebox{7cm}{8cm}
			{\includegraphics{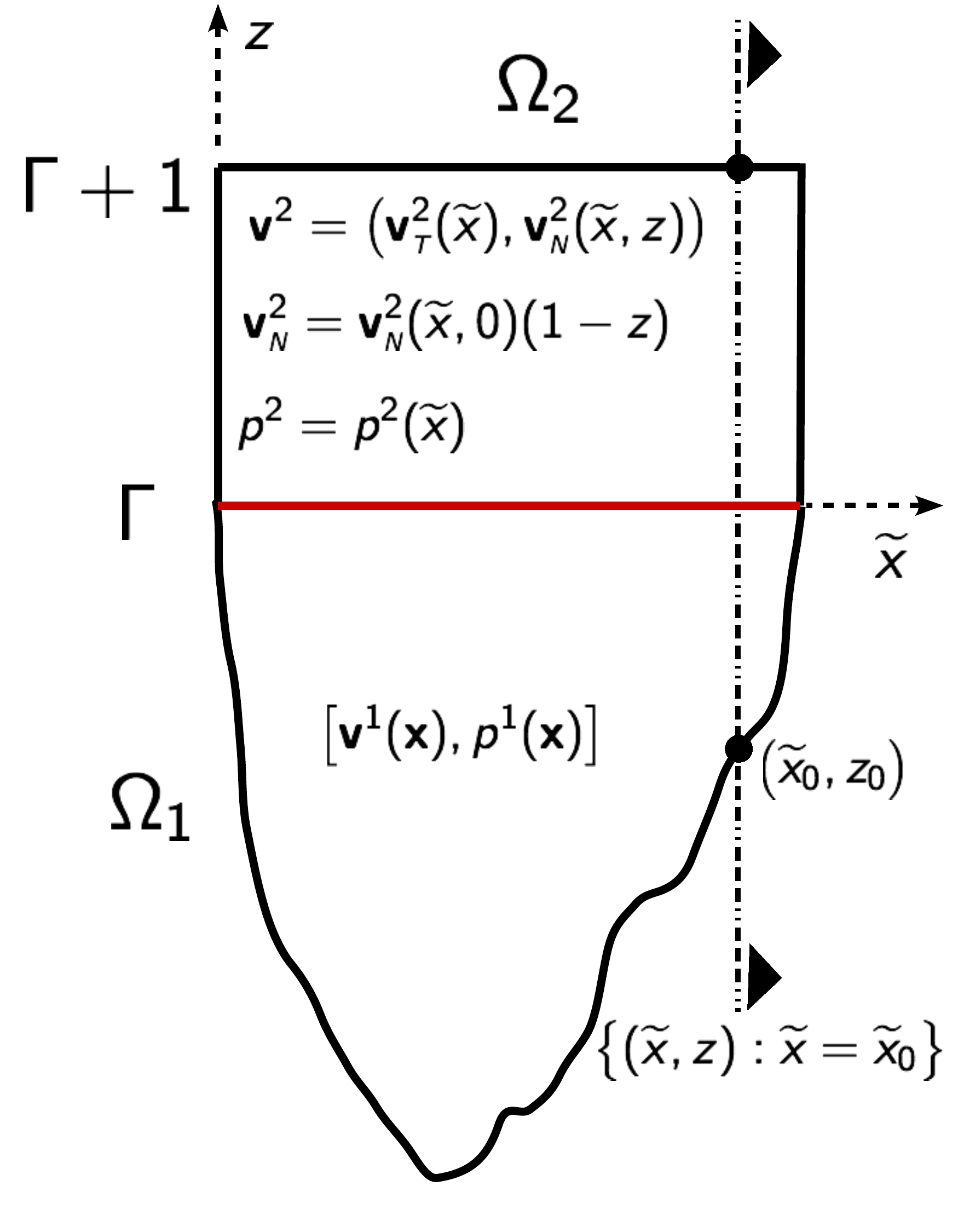} } }
	\end{subfigure} 
	\qquad
	~ 
	\begin{subfigure}[Velocity and Pressure Schematic Traces for the Solution on the hyperplane $\big\{ (\widetilde{x}, z):\widetilde{x} = \widetilde{x}_{0} \big\}$.]
		{\resizebox{7cm}{8cm}
			{\includegraphics{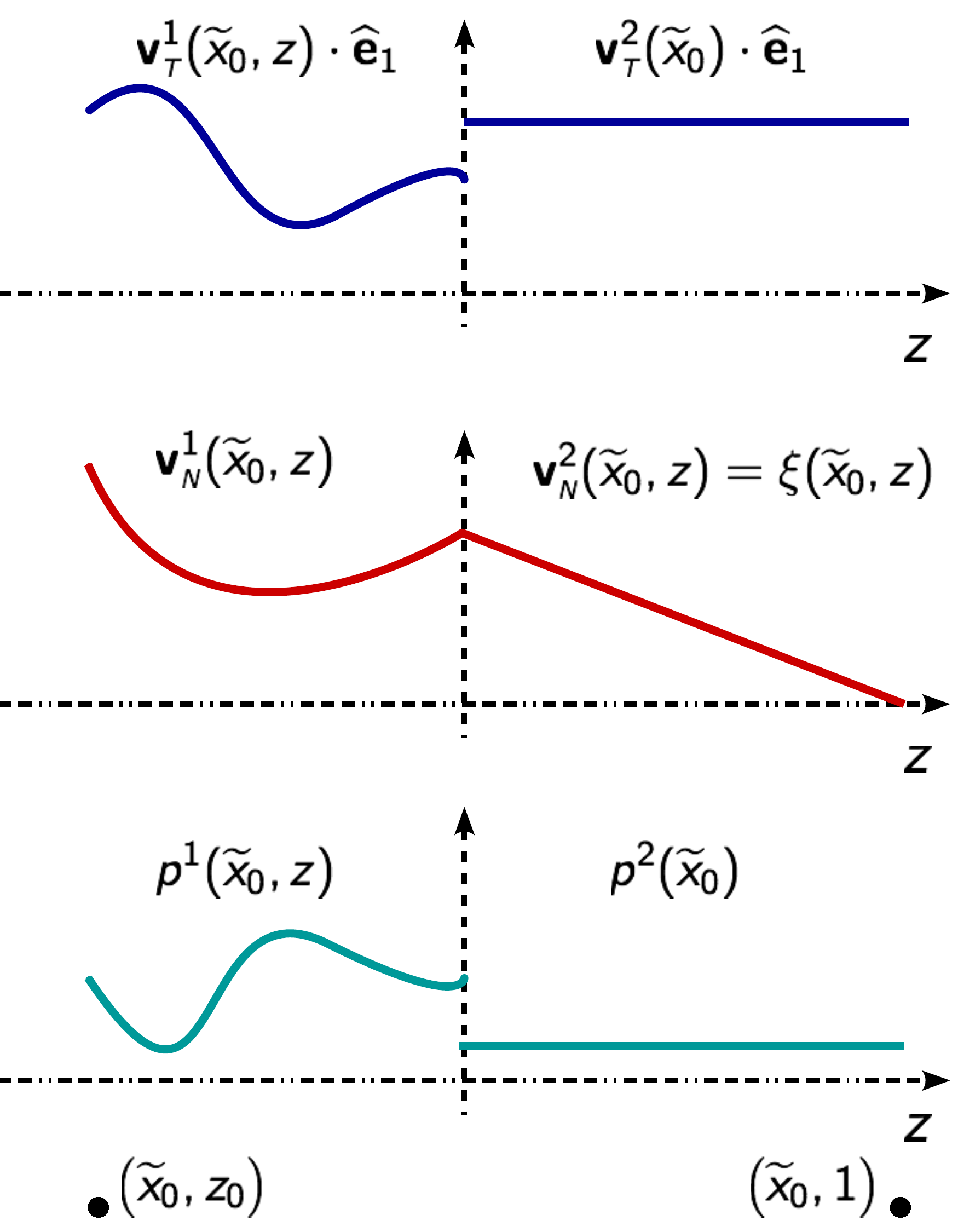} } }                
	\end{subfigure} 
	%
	%
	%
	\caption{Figure (a) shows the dependence of the limit solution $\big[\u, p\big]$ according to the respective region. Figure (b) depicts some plausible schematics of the velocity and pressure traces on the hyperplane $\big\{ (\widetilde{x}, z):\widetilde{x} = \widetilde{x}_{0} \big\}$.}\label{Fig Limit Solution Schematics}
\end{figure}
%
\subsection{Dimension Reduction}
The limit tangential velocity and pressure in $\Omega_2$ are
independent of $x_{\scriptscriptstyle N}$ (see
\eqref{solutionconvergence} and \eqref{Eq Dependence of xi and
Pressure 2}). Consequently, the spaces $\X^0$, $\Y^0$ and the problem
\eqref{formulation limit problem} can be dimensionally reduced to
yield a coupled problem on $\Omega_1 \times \Gamma$. To that end, we
first modify the function spaces. For the pressures we define
the space
\begin{equation}\label{reduced space of the limit pressure}
\Y^{00}\defining \Big\{
\big [\varphi^{1}, \varphi^{2} \big] \in 
L^{2}(\Omega_{1})\times L^{2}(\Gamma): 
\int_{\Omega_{1}}\varphi^{1} 
+ \int_{\Gamma}\varphi^{2}d\xthilde = 0  \Big\} ,
\end{equation}
endowed with its natural norm. For the velocities we define the space
\begin{subequations}\label{Def reduced space of the limit velocity}
\begin{equation}\label{reduced space of the limit velocity}
\X^{00}\defining
\Big\{\,\big(\wone,\wtwo \big)\in \Hdiv(\Omega_{1})\times
\big(H_{0}^{1}(\Gamma)\big)^{2}:
\wone\cdot\n \Big\vert_{\Gamma}\in L^{2}(\Gamma)\,\Big\} ,
\end{equation}
with the norm
\begin{equation}\label{norm reduced space of the limit velocity}
\big\Vert \big[\wone, \wtwo \big]\big\Vert_{\X^{00} }
\defining \Big(\big\Vert \wone\big\Vert_{\Hdiv}^{2} 
+ \big\Vert \wone\cdot \n \big\Vert_{0, \Gamma}^{2}
+ \big\Vert\wtwo \big\Vert_{1, \Gamma} 
\Big)^{1/2} .
\end{equation}
%
\end{subequations}
\begin{remark}\label{Rem Isomorphism with Lower Dimensional Spaces}
Clearly the pressure spaces $\Y^{00}$ and $\Y^0$ are isomorphic (see
Figure \ref{Fig Limit Solution Schematics}). It is also direct to see
that the application $\iota: \X^0\rightarrow \X^{00}$, given by
$\big[\wone, \wtwo \big]\mapsto \big[\wone, \wtan \big]$, is an
isomorphism, because $\wnorm$ is entirely determined by its trace on
$\Gamma$ and $\wnorm(\widetilde{x}, 0) = \wone\cdot \n$ (see Figure
\ref{Fig Limit Solution Schematics}).
\end{remark}
\begin{theorem}\label{Th Lower Dimensional Limiting Problem}
Let $[\v, p]$, with $\vtwo = \big[\vtan, \xi\big]$, be the weak limits found in Corollary \ref{Th Direct Weak Convergence of Velocities} and Lemma \ref{Th Convergence of Pressure One}. Then, making the corresponding identifications, $[\v, p]$ satisfies the following variational statement
\begin{subequations}\label{lower dimensional limit problem}
\begin{flushleft}
$[\,\v,p\,]\in \X^{00}\times \Y^{00}:$
\end{flushleft}
\vspace{-0.4cm}
\begin{multline} \label{lower dimensional limit problem 1}
\int_{\Omega_1}  \Q\, \vone \cdot \wone 
- \int_{\Omega_1}\pone \,\grad\cdot\wone 
+ (\,\E+\alpha\,)\,\int_{\Gamma}\big(\,\vone\cdot\n\,\big) 
\,\big(\,\wone\cdot\n\,\big)\,\,d\widetilde{x} ,
\\
+\int_{\Gamma} \ptwo \,\big(\,\wone\cdot\n\,\big)\,\,d\widetilde{x} 
-  \int_{\Gamma}\ptwo \, \grad_{\!T}\cdot\wtan \,d\widetilde{x}
+ \int_{\Gamma} \beta \, \sqrt{\Q} \,\,\vtan\, \cdot \wtan \, d \widetilde{x}
 +\int_{\Gamma}\E\,\grad_{\!T}\,\vtan:\grad_{\!T}\,\wtan \, d\widetilde{x}
 = \int_{\Gamma} {\mathbf f}^{2}_{T}
\,\cdot\, \wtan \,d\widetilde{x} ,
\end{multline}
\begin{equation}  \label{lower dimensional limit problem 2}
\int_{\Omega_1}\grad\cdot\vone\, \varphi^1 
-\int_{\Gamma} \, \big(\,\vone\cdot\n\,\big)\, \varphi^2  \,d\widetilde{x} 
+ \int_{\Gamma} \grad_{\!T}\cdot\,\vtan\, \varphi^2 \,d\widetilde{x}
= \int_{\Omega_1} h^{1} \, \varphi^1 , 
\end{equation}
\begin{flushright}
for all $[\,\w,\varphi \,]\in \X^{00}\times \Y^{00}$.
\end{flushright}
\end{subequations}
Furthermore, the mixed formulation of the Problem above is given by 
%
\begin{equation}\label{reduced limit problem mixed formulation}
\begin{split}
[\,\v,p\,]\in \X^{00}\times\Y^{00}:
\A\,\v - \B ' p &=\mathbf{f} , \\
%
\B\,\v  & = h ,
\end{split}
\end{equation}
with the forms $\A: \X^{00} \rightarrow (\X^{00})'$ and $\B :
\X^{00}\rightarrow (\Y^{00})'$ defined by
\begin{subequations}\label{Def reduced limit operators}
\begin{equation}\label{reduced limit operator A}
\A \defining
\left(\begin{array}{cc}\, \Q + \gamma_{\, \n}
'\,(\,\E+\alpha\,)\,\gamma_{\, \n}  
& \0 
\\
\0 & \,\,\beta\,\sqrt{\Q}
+\,(\grad_{\,T})'\,\E\,\grad_{\! T}\end{array}
\right) ,
\end{equation}
\begin{equation}\label{reduced limit operator B }
\B \defining
\left(\begin{array}{cc} \,\grad\cdot & 0 \\
-\gamma_{\, \n} &  \grad_{\!T}\cdot \end{array} \right)
=\left(\begin{array}{cc} \,\Div & 0 \\
-\gamma_{\, \n} &  \Div_{\scriptscriptstyle T}  \end{array} \right) .
\end{equation}
\end{subequations}
\end{theorem}
\begin{proof}
Notice that if $\wtwo\in\X^0$ then $\partial_z\,\wnorm
=\partial_z\,\wnorm\, (\widetilde{x})= -\wnorm(\widetilde{x}, 0) = - \wtwo\cdot\n = -\wone\cdot\n$ (see Figure \ref{Fig Limit Solution Schematics} ).
Next, we introduce this observation in the Statement \eqref{formulation limit problem} above, together with the Identity \eqref{boundary conditions on normal velocity}; this gives
\begin{subequations} \label{simplified limit problem}
\begin{flushleft}
$[\,\v,p\,]\in \X^0\times \Y^0: $
\end{flushleft}
\vspace{-0.4cm}
\begin{multline}\label{simplified limit problem 1}
\int_{\Omega_1}  \Q\, \vone \cdot \wone  
- \int_{\Omega_1}\pone \, \grad\cdot\wone 
+ (\E+\alpha)\,\int_{\Gamma}\big(\,\vone\cdot\n\,\big) \,\big(\,\wone\cdot\n\,\big)\,\,d\widetilde{x}
\\
+\int_{\Gamma} \ptwo\,\big(\,\wone\cdot\n\,\big)\,\,d\widetilde{x} 
-  \int_{\Gamma}\ptwo \, \grad_{\!T}\cdot\wtan \,d\widetilde{x}
%
+\int_{\Gamma}\E\,\grad_{\!T}\,\vtan:\grad_{\!T}\,\wtan\;d\widetilde{x}
+ \int_{\Gamma} \beta \,  \sqrt{\Q} \,\,\vtan\, \cdot \w_T^2\,d S 
= \int_{\Gamma} {\mathbf f}^{2}_{T}\cdot \wtan \,d\widetilde{x} ,
\end{multline}
\begin{equation}\label{simplified limit problem 2}
\int_{\Omega_1}\grad\cdot\vone\, \varphi^1 
-\int_{\Gamma} \, \big(\,\vone\cdot\n\,\big)\, \varphi^{2} \,d\widetilde{x} 
+ \int_{\Gamma} \grad_{\!T}\cdot\,\vtan\, \varphi^{2} \,d\widetilde{x}
= \int_{\Omega_1} h^{1} \, \varphi^1 , 
\end{equation}
\begin{flushright}
for all $[\,\w,\varphi \,]\in \X^0\times \Y^0 $.
\end{flushright}
\end{subequations}
Due to the isomorphism between spaces as highlighted in Remark \ref{Rem Isomorphism with Lower Dimensional Spaces}, the result follows. 
\qed
\end{proof}
%
Now we outline an independent proof that Problem \eqref{reduced limit
problem mixed formulation} is well-posed.
\begin{theorem}\label{Th well-posedness of the reduced limiting problem}
The problem \eqref{reduced limit problem mixed formulation} is well-posed.
\end{theorem}
\begin{proof}
First, showing that the form $\v\mapsto \A(\v) \v$ is $\X^{00}$-elliptic, is identical to the proof of Lemma \ref{Th coercivity of A in the limit}. Next, proving that $\B$ has closed range follows exactly as the proof of Lemma \ref{Th B closed range in the limit} with only one extra observation. Notice that defining $\uone\defining -\grad \phi$, where $\phi$ is the solution of the auxiliary Problem \eqref{Pblm Auxiliary Problem}, satisfies $\u\cdot \n \big\vert_{\Gamma}= \varphi^{2}\in L^{2}(\Gamma)$ then, recalling the Estimate \ref{Ineq Estimate on the Auxiliary function one} we get
\begin{equation*}
\big\Vert \uone \big\Vert_{\Hdiv(\Omega_{1})}^{2}
+ \big\Vert \uone \cdot \n\big\Vert_{0, \Gamma}^{2}
\leq 
C_{1}^{2}\, \big(\Vert \varphi^{1} \Vert_{0,\Omega_{1}}^{2} 
+ \Vert \varphi^{2} \Vert_{0, \Omega_{2}}^{2} \big)
+ \Vert \varphi^{2} \Vert_{0, \Omega_{2}}^{2} 
\leq C \big(\Vert \varphi^{1} \Vert_{0,\Omega_{1}}^{2} 
+ \Vert \varphi^{2} \Vert_{0, \Omega_{2}}^{2} \big) .
\end{equation*}
From here, it follows trivially that the function
$\u =\big(\u^1, \0_{T}\big)$ satisfies an estimate of the type \eqref{Ineq Estimate on the whole Auxiliary function} as well as a chain of inequalities analogous to \eqref{Ineq Chain of Inequalities for Closed Range}. Therefore, the operators $\A$ and $\B$ satisfy the hypotheses of Theorem \ref{Th well posedeness mixed formulation classic} and the Problem \eqref{reduced limit problem mixed formulation} is well-posed.
\qed
\end{proof}

We close this section with the strong form of Problem \eqref{lower
dimensional limit problem}.
\begin{proposition}\label{Th Limit Problem Strong Form}
The problem \eqref{lower dimensional limit problem} is the weak
formulation of the boundary-value problem 
\begin{subequations}\label{Pblm limit strong problem fixed geometry}
\begin{equation}
\Q \,\vone +\grad p^{1}=0\,,
\end{equation}
\begin{equation}  
\grad\cdot\vone = h^{\,1}\quad\text{in}\quad\Omega_{1} ,
\end{equation}
\begin{equation}  
\grad_{T}\, \ptwo
+\beta\,\sqrt{\Q}\,\vtan-\big(\grad_{\,T}\big)'\E\,\grad_{\,T}\big(\vtan\big) =
\f^{2}_{T}\,,
\end{equation}
\begin{equation}  
\grad_{T}\cdot\vtan
- \vone\cdot\n= 0\,,
\end{equation}
\begin{equation}  
\pone-\ptwo
=(\,\E+\alpha\,)\,\vone\cdot\n\quad\text{in}\quad\Gamma ,
\end{equation}
\begin{equation}  
\pone=0\quad\text{on}\quad\partial\Omega_{1} -\Gamma ,
\end{equation}
\begin{equation}  
\vtan= \0\quad\text{on}\quad\partial\Gamma .
\end{equation}
\end{subequations}
\end{proposition}
%
The problem \eqref{Pblm limit strong problem fixed geometry} is
obtained using the standard decomposition of weak deriatives in
Problem \eqref{lower dimensional limit problem} to get the
differential equations in the interior and then the boundary and
interface conditions.
%
%
\section{Strong Convergence of the Solutions}  \label{strong}
%
%
In this section we show the strong convergence of the velocities and pressures to that of the limiting Problem \eqref{formulation limit problem}. The strategy is the standard approach in Hilbert spaces: given that the weak convergence of the solutions $\big[\veps, \peps\big]\xrightarrow[\epsilon \, \rightarrow \, 0]{} [\v, p] $ holds, it is enough to show the convergence of the norms in order to conclude strong convergence statements. Before showing these results a further hypothesis needs to be accepted.
\begin{hypothesis}\label{Hyp Strong Convergence on the Forcing Terms}
In the following, it will be assumed that the sequence of forcing terms $\{\f^{2,\epsilon}: \epsilon > 0\}\subseteq \mathbf{L}^{2}(\Omega_{2})$ and $\{h^{1,\epsilon}: \epsilon > 0\}\subseteq L^{2}(\Omega_{1})$ are strongly convergent i.e., there exist $\f^{2} \in \mathbf{L}^{2}(\Omega_{2})$ and $h^{1}\in L^{2}(\Omega_{1})$ such that
\begin{align}\label{Ineq Strong Convergence on the Forcing Terms}
& \big\Vert \f^{2, \epsilon} - \f^{2}\big\Vert_{0, \Omega_{2} }
\xrightarrow[\epsilon \,\rightarrow \, 0]{}   0,&
& \big\Vert h^{1, \epsilon} - h^{1}\big\Vert_{0, \Omega_{1} } 
\xrightarrow[\epsilon \,\rightarrow \, 0]{} 0\, .
\end{align}
\end{hypothesis}
\begin{theorem}\label{Th Strong Convergence of Velocities}
Let $\big\{\big[\veps, \peps \big]: \epsilon > 0 \big\}\subseteq \X\times \Y$ be the sequence of solutions to the family of Problems \eqref{problem fixed geometry} and let $[\v, p]\in \X^0\times \Y^0$, with $\vtwo = [\vtan, \xi]$, be the solution of Problem \eqref{formulation limit problem}, then 
\begin{subequations}\label{Stmt strong convergence velocities}
\begin{align} \label{strong convergence v tangential}
& 
\big\Vert \vtaneps -\vtan\big\Vert_{0,\Omega_2}\rightarrow 0 ,&
& \big\Vert \grad_{\!T}\,\vtaneps -\grad_{\!T}\,\vtan\big\Vert_{0,\Omega_2}  \rightarrow 0 .
\end{align}
\begin{equation} \label{strong convergence v normal}
\big\Vert \vnormeps-\vnorm\big\Vert_{\Hpartial}\rightarrow 0 .
\end{equation}
%
%
%
%
\begin{equation} \label{strong convergence v omega 1}
\big\Vert \vepsone-\vone\big\Vert_{\Hdiv(\Omega_1)}\,\rightarrow 0 .
\end{equation}
\end{subequations}
\end{theorem}
\begin{proof}
In order to prove the convergence of norms, a new norm on the space $\X_2^0$, defined in \eqref{space of the limit velocity omega 2}, must be introduced
\begin{multline}\label{adequated norm}
\w\mapsto \Big\{\,
\Vert \sqrt{\E}\,\grad_{\!T}\,(\,\wtan) \Vert_{0,\Omega_2}^{2}
+\Vert \sqrt{\E}\,\partial_z\,\wtan \Vert_{0,\Omega_2}^{2}
+\Vert \sqrt{\E}\,\partial_z\,\wnorm \Vert_{0,\Omega_2}^{2} 
+\Vert \sqrt{\alpha}\,\wnorm \Vert_{0,\Gamma}^{2}
+\Vert \sqrt{\beta}\,\sqrt[4]{\Q}\,\wtan\, \Vert_{0,\Gamma}^{2}\,\Big\}^{1/2}\\
\defining \big\Vert \w\big\Vert_{\,\X_2^0} \, .
\end{multline}
Clearly, this norm is equivalent to the $\Vert \cdot \Vert_{\X_2^0}$-norm defined in \eqref{norm space of the limit velocity omega 2}. Now consider
\begin{multline} \label{epsilon diagonal}
\limsup_{\,\epsilon\,\downarrow\,0}\,
\Big\{ \Vert \sqrt{\Q}\,\vepsone \Vert_{0,\Omega_{1}}^{2} 
+ \big\Vert \big[\,\vtan,\,\xi\,\big]\big\Vert_{\X_2^0}^{2}
\Big\} \\
\leq 
\limsup_{\,\epsilon\,\downarrow\,0}\,
\Big\{ \Vert \sqrt{\Q}\,\vepsone \Vert_{0,\Omega_{1}}^{2}
%
+\Vert \sqrt{\E}\,\grad_{\!T}\,\big(\epsilon\,\vtaneps\,\big) \Vert_{0,\Omega_2}^{2}
%
+\Vert \sqrt{\E}\,\partial_z\,\vnormeps \Vert_{0,\Omega_2}^{2}
%
+\Vert \sqrt{\E}\,\big(\epsilon\,\grad_{\!T}\,\vnormeps
\big) \Vert_{0,\Omega_2}^{2} \\
+\Vert \sqrt{\E}\,\big(\partial_z\,\vtaneps\,\big) \Vert_{0,\Omega_2}^{2} 
+\Vert \sqrt{\alpha}\,\vepsone\cdot\n \Vert_{0,\Gamma}^{2}
+\Vert \sqrt{\beta}\,\sqrt[4]{\Q}\,\big(\,\epsilon\,\vtaneps\,\big) \Vert_{0,\Gamma}^{2}\,\Big\}
%
 \leq \, \int_{\Omega_2} {\mathbf f^{2}} \cdot \vtwo 
+\int_{\Omega_1} h^{1}  \, \pone . 
\end{multline}
On the other hand, testing the Equations \eqref{formulation limit problem} on the solution $[\v, p]$ and adding them together, gives
\begin{multline} \label{diagonal}
\Vert \sqrt{\Q}\,\vone \Vert_{0,\Omega_{1}}^{2}
+\Vert \sqrt{\E}\,\grad_{\!T}\,\vtan \Vert_{0,\Omega_2}^{2}
+\Vert \sqrt{\E}\,\partial_z\,\xi \Vert_{0,\Omega_2}^{2}\\
+\Vert \sqrt{\alpha}\,\vone\cdot\n \Vert_{0,\Gamma}^{2}
+\Vert \sqrt{\beta}\,\sqrt[4]{\Q}\,\vtan\, \Vert_{0,\Gamma}^{2} 
= \, \int_{\Omega_2} {\mathbf f^{2}} \cdot
\vtwo \,d\widetilde{x} \,dz +\int_{\Omega_1} h^{1} \,
\pone \, dx .
\end{multline}
Comparing the left hand side of \eqref{epsilon diagonal} and \eqref{diagonal} we conclude one inequality. 

Next, due to the weak convergence of the sequence $\{\,[\,\epsilon\,\vtaneps, \vnormeps]: \epsilon > 0\,\}\subseteq \X_2^0$, it must hold that
\begin{equation} \label{lower inequality, velocity 2}
\begin{split}
\big\Vert \big[\,\vtan,\,\xi\,\big]\big\Vert_{\X_2^0}^{2}
\leq 
\liminf_{\,\epsilon\,\downarrow\,0} 
\big\Vert  \,\big[\,\epsilon\,\vtaneps, \vnormeps\big]\big\Vert_{\X_2^0}^{2} 
=\liminf_{\,\epsilon\,\downarrow\,0}\,& \Big\{   \Vert  \,\sqrt{\E}\,\grad_{\!T}\,(\,\epsilon\,\vtaneps) \Vert_{0,\Omega_2}^{2}
+\Vert \sqrt{\E}\,\partial_z\,\vtaneps \Vert_{0,\Omega_2}^{2} \\
& +\Vert \sqrt{\E}\,\partial_z\,\vnormeps \Vert_{0,\Omega_2}^{2}
+\Vert \sqrt{\alpha}\,\vnormeps \Vert_{0,\Gamma}^{2}
+\Vert \sqrt{\beta}\,\sqrt[4]{\Q}\,\vtaneps\, \Vert_{0,\Gamma}^{2}\,\Big\} .
\end{split}
\end{equation}
In addition, due to the weak convergence discussed in Corollary
\ref{Th Weak convergence of the whole sequence}, in particular it
holds that
\begin{equation} \label{lower inequality, velocity 1}
\big\Vert \sqrt{\Q} \,  \vone\big\Vert_{0,\Omega_{1}}^{2}\leq \liminf_{\,\epsilon\,\downarrow\,0}
\big\Vert \sqrt{\Q} \,  \vepsone\big\Vert_{0,\Omega_{1}}^{2} .
\end{equation}
Putting together \eqref{lower inequality, velocity 2}, \eqref{lower
inequality, velocity 1}, \eqref{epsilon diagonal} and \eqref{diagonal}
we conclude that the norms are convergent, i.e.
\begin{equation} \label{norms limit}
\big\Vert (\,\vone, [\,\vtan,\,\xi\,]\,)\big\Vert_{\,\mathbf{L}^{2}(\Omega_1)\times\X_2^0}^{2}
= \lim_{\,\epsilon\,\downarrow\,0}\big\Vert (\,\vepsone, [\,\epsilon\,\vtaneps,\,\vnormeps\,]\,)\big\Vert_{\,\mathbf{L}^{2}(\Omega_1)\times\X_2^0}^{2} \, .
\end{equation}
%
%
%
Since the norm $\w\mapsto \Vert\sqrt{\Q} \, \w\Vert_{0, \Omega_{1}}$
is equivalent to the standard $\mathbf{L}^{2}(\Omega_{1})$-norm, we
conclude the strong convergence of the sequence $\big\{\big(\vepsone,
\big[\epsilon\,\vtaneps,\,\vnormeps\big]\big):\epsilon>0\big\}$ to
$[\v, p]$ as elements of $ \mathbf{L}^{2}(\Omega_{1})\times\X_2^0$. In
particular, the Statements \eqref{strong convergence v tangential} and
\eqref{strong convergence v normal} follow. Finally, recalling the
Equality \eqref{mass conservation Omega 1} and the strong convergence
of the forcing terms $\{h^{1, \epsilon}: \epsilon> 0 \}$, the
Statement \eqref{strong convergence v omega 1} follows and the proof
is complete.
\qed
\end{proof}
\begin{remark}\label{Rem Strong Convergence of the Vanishing Velocities}
Notice that \eqref{norms limit} together with \eqref{epsilon diagonal} imply
\begin{equation}\label{twisted term}
c\,\lim_{\,\epsilon\,\downarrow\,0}\,
\Big\{\big\Vert\grad_{\!T}\big(\,\epsilon\,\vnormeps\big)\big\Vert_{0,\Omega_2}^{2}
+\big\Vert\partial_z\,\vtaneps\big\Vert_{0,\Omega_2}^{2} \Big\}
%
\leq\lim_{\,\epsilon\,\downarrow\,0}\,
\Big\{\big\Vert\E\,\grad_{\!T}\big(\,\epsilon\,\vnormeps\big)\big\Vert_{0,\Omega_2}^{2}
+\big\Vert \E\,\partial_z\,\vtaneps\big\Vert_{0,\Omega_2}^{2}\Big\}=0 ,
\end{equation}
where $c>0$ is an ellipticity constant coming from $\E$.
\end{remark}

Next, we show the strong convergence of the pressures.
\begin{theorem}\label{Th Strong Convergence of the Pressures}
Let $\{\big[\veps, \peps \big]: \epsilon > 0 \}\subseteq \X\times \Y$
be the sequence of solutions to the family of Problems \eqref{problem
fixed geometry} and let $[\v, p]\in \X^0\times \Y^0$, be the solution
of Problem \eqref{formulation limit problem}, then
\begin{subequations}\label{Eq Strong Convergence of Pressures}
\begin{equation}\label{strong convergence presssure omega 1}
\big\Vert\pepsone - \pone \big\Vert_{1,\Omega_{1}}\rightarrow 0 ,
\end{equation}
\begin{equation}\label{strong convergence of the epsilon pressure in L2}
\big\Vert \pepstwo - \ptwo \big\Vert_{0,\Omega_2}
\rightarrow 0 .
\end{equation}
\end{subequations}
\end{theorem}
\begin{proof}
For the Statement \eqref{strong convergence presssure omega 1} first
observe that \eqref{strong convergence v omega 1} together with
\eqref{Darcy's law Strong problem} imply $\big\Vert \grad\pepsone -
\grad \pone\big\Vert_{0,\Omega_{1}}\rightarrow 0$. Again, since
$\pepsone =0$ on $\partial \Omega_{1} - \Gamma$ then $\pone=0$ on
$\partial \Omega_{1} - \Gamma$ and, due to Poincar\'e's inequality,
the gradient controls the $H^{1}(\Omega_{1})$-norm, of $\pone$ and
$\pepsone$ for all $\epsilon > 0$. Consequently, the convergence
\eqref{strong convergence presssure omega 1} follows.

Proving the Statement \eqref{strong convergence of the epsilon
pressure in L2} is significantly more technical. We start taking a
previous localization step for the function $\pepstwo$. Let
$\phi_{\epsilon}\in C_{0}^{\infty}(\Omega_{2})$ be a function such
that
\begin{equation*} 
\big\Vert \pepstwo-\phi_{\epsilon} \big\Vert_{0,\Omega_2}<\epsilon .
\end{equation*}
Observe that
\begin{equation} \label{approximation of pressure norm by local functions}
\Big\vert\,\int_{\Omega_2}\pepstwo\,\pepstwo 
-
\int_{\Omega_2}\pepstwo\,\phi_{\epsilon} 
\Big\vert
=\Big\vert\int_{\Omega_2}\pepstwo\,\big(\,\pepstwo-\phi_{\epsilon}\big) 
\Big\vert
\leq \big\Vert \pepstwo \big\Vert_{0,\Omega_2}\,
\big\Vert \pepstwo-\phi_{\epsilon} \big\Vert_{0,\Omega_2}
< \widetilde{C}\,\epsilon ,
\end{equation}
where the last inequality in the expression above holds due to the
Statement \eqref{convergence of the pressure in Omega_2}. In addition
$\phi_{\epsilon}\rightarrow \ptwo$ weakly in $L^{2}(\Omega_{2})$
because, for any $w\in L^{2}(\Omega_{2})$ it holds that
\begin{equation*} 
\int_{\Omega_2}\phi_{\epsilon}\,w 
=\int_{\Omega_2}\big(\phi_{\epsilon}-\pepstwo\big)\,w 
+\int_{\Omega_2}\pepstwo\,w 
\\
\rightarrow 0
+\int_{\Omega_2}\ptwo\,w  .
\end{equation*}
In particular, taking $w=w(\widetilde{x})$ in the expression above, we
conclude that $\int_{0}^{1}\phi_{\epsilon}\,dz\rightarrow \ptwo$
weakly in $L^{2}(\Gamma)$.

Now, for $\phi_{\epsilon}$ define the function $\varsigma_{\,\epsilon}$ using the rule presented in the Identity
\eqref{negative antiderivative}, therefore
$\varsigma_{\,\epsilon} \,\vert_{\,\Gamma} = \varsigma_{\,\epsilon}(\widetilde{x},0)=\int_{0}^{1}\phi_{\epsilon}\,dz$ belongs to $L^{2}(\Gamma)$ and, by construction, $\varsigma_{\,\epsilon} \,\vert_{\,\Gamma}$ is bounded in $L^{2}(\Gamma)$. Then, due to Lemma \ref{Th Surjectiveness from Hdiv to H^1/2}, there must exist
$\wone_{\epsilon}\in\Hdiv(\Omega_{1})$ such that
$\wone_{\epsilon}\cdot\n=\int_{0}^{1}\phi_{\epsilon}\,dz$ on
$\Gamma$, $\wone_{\epsilon}\cdot\n = 0$ on $\partial\,\Omega_1-\Gamma$ and
$\big\Vert \wone_{\epsilon}\big\Vert_{\,\Hdiv(\Omega_1)}\leq C_{1}
\big\Vert \varsigma_{\,\epsilon}(\widetilde{x},0)\big\Vert_{\,\Gamma}<C$. Where $C_{1}$ depends only on the domain. Hence, the function $\w_{\epsilon}\defining [\,\wone_{\epsilon},\w^{2}_{\epsilon}\,]$,
with $\w^{2}_{\epsilon}\defining\big( \0_{\,T},\,\varsigma_{\,\epsilon}(\widetilde{x},z)\big)$, belongs to $\X$. Test, \eqref{problem fixed geometry 1} with $\w_{\epsilon}$ and get
\begin{multline}\label{relation pi, xi alpha p one on the diagonal}
\int_{\Omega_1} \Q\,\vepsone\cdot\wone_{\epsilon} 
- \int_{\Omega_1} \pepsone\,\grad\cdot\wone_{\epsilon} 
+\alpha\int_{\,\Gamma}\,\big(\vepsone\cdot\n\big)\big(\wone_{\epsilon}\cdot\n\big)\,d\widetilde{x}\\
+\int_{\Omega_2}
\pepstwo\,\phi_{\epsilon}(\widetilde{x},z)\,d\widetilde{x}\,dz
+\epsilon\int_{\Omega_2}\E\,\grad_{\,T}\big(\epsilon\,\vnormeps
\big)\cdot\,\grad_{\,T}\,\varsigma_{\,\epsilon}(\widetilde{x},z)\,d\widetilde{x}\,dz\\
-\int_{\Omega_2}\E\,\partial_z\,\vnormeps\,\phi_{\epsilon}(\widetilde{x},z)\,d\widetilde{x}\,dz
=\epsilon\,\int_{\Omega_2}\f^{\,2,\,\epsilon}_{N}\,\varsigma_{\,\epsilon}\,d\widetilde{x}\,dz .
\end{multline}
In the Identity \eqref{relation pi, xi alpha p one on the diagonal}
all the summands but the fourth, are known to be convergent due to the
previous strong convergence statements, therefore, this last summand
must converge too. The first two summands satisfy
\begin{equation*}
\int_{\Omega_1} \Q\,\vepsone\cdot\w^{1}_{\epsilon} 
- \int_{\Omega_1} \pepsone \grad\cdot\wone_{\epsilon} 
= -\int_{\,\Gamma}\pepsone \, \big(\,\wone_{\epsilon}\cdot\n\,\big)\,d\widetilde{x}
%
= -\int_{\,\Gamma}\pepsone \, 
\Big(\,\int_{0}^{1}\phi_{\epsilon}\,dz\,\Big)\,d\widetilde{x}
\rightarrow -\int_{\,\Gamma}\pone \, \ptwo\,d\widetilde{x} .
\end{equation*}
The limit above holds due to the strong convergence of the pressure in
$H^1(\Omega_{1})$ and the weak convergence of $\int_{[\,0,1
]}\phi_{\epsilon}\,dz$. The third summand in \eqref{relation pi, xi
alpha p one on the diagonal} behaves as
\begin{equation*}
\alpha\int_{\,\Gamma}\,\big(\,\vepsone\cdot\n\,\big)\,
\big(\,\wone_{\epsilon}\cdot\n\,\big)\,d\widetilde{x}
\\
=\int_{\Gamma}\alpha\,\big(\,\vepsone\cdot\n\,\big)\,\big(\,\int_{0}^{1}\phi_{\epsilon}\,dz\,\big)\,d\widetilde{x}
\rightarrow \int_{\Gamma}\alpha\,\big(\,\vone\cdot\n\,\big)\,\ptwo\,d\widetilde{x} ,
\end{equation*}
because of the Statement \eqref{strong convergence v normal}. Next,
the fifth summand in \eqref{relation pi, xi alpha p one on the
diagonal} vanishes due to the Estimates \eqref{twisted term} and the
sixth summand behaves in the following way
\begin{multline*}
-\int_{\Omega_2}\E\,\partial_z\,\vnormeps\,\phi_{\epsilon}(\widetilde{x},z)\,d\widetilde{x}\,dz\rightarrow
-\int_{\Omega_2}\E\,\partial_z\,\xi\,\big(\wklim_{\,\epsilon\,\downarrow\,0}\,\phi_{\epsilon}(\widetilde{x},z)\,\big)\,d\widetilde{x}\,dz\\
=-\int_{\Gamma}\E\,\partial_z\,\xi\,
\Big(\int_{\,0}^{1}\wklim_{\,\epsilon\,\downarrow\,0}\,\phi_{\epsilon}(\widetilde{x},z)\,dz\,\Big)
\,d\widetilde{x} =-\int_{\Gamma}\E\,\partial_z\,\xi\,\ptwo\,d\widetilde{x} .
\end{multline*}
The first equality above holds since $\partial_z\,\xi =
\partial_z\,\xi (\widetilde{x})$, while the second holds, because
$\int_{\,0}^{1}\wklim\limits_{\,\epsilon\,\downarrow\,0}\,\phi_{\epsilon}(\widetilde{x},z)\,dz=
\wklim\limits_{\,\epsilon\,\downarrow\,0}\,\int_{\,0}^{1}\phi_{\epsilon}(\widetilde{x},z)\,dz
= \ptwo$. Finally, the right hand side on \eqref{relation pi, xi alpha
p one on the diagonal} vanishes. Putting together all these
observations we conclude that
\begin{equation*}
\int_{\Omega_2}
\pepstwo\,\phi_{\epsilon}(\widetilde{x},z)\,d\widetilde{x}\,dz
\rightarrow \int_{\Omega_2}\big(\,\E\,\partial_z\,\xi - \alpha\,\vone\cdot\n\,\vert_{\,\Gamma}+\pone\,\vert_{\,\Gamma}\,\big)\,\ptwo\,d\widetilde{x} .
\end{equation*}
The latter, together with \eqref{approximation of pressure norm by
local functions} and \eqref{second identification of xi} imply
\begin{equation*}
\big\Vert  \pepstwo\big\Vert_{0,\Omega_2}^{2}
\rightarrow \int_{\Omega_2}\big(\,\E\,\partial_z\,\xi 
- \alpha\,\vone\cdot\n\,\vert_{\,\Gamma}+\pone\,\vert_{\,\Gamma}\,\big)\,\ptwo\,d\widetilde{x}
=\int_{\Gamma}\ptwo\,\ptwo
\,d\widetilde{x}=\big\Vert \ptwo\big\Vert_{0,\Omega_2}^{2} .
\end{equation*}
Again, the convergence of norms together with the weak convergence of
the solutions stated in Corollary \ref{Th Weak convergence of the
whole sequence}, imply the strong convergence Statement \eqref{strong
convergence of the epsilon pressure in L2}.
\qed
\end{proof}
%
%
%
%
%
\subsection{Comments on the Ratio of Velocities}
The ratio of velocity magnitudes in the tangential and the normal directions is very high and tends to infinity as expected. Since $\{ \Vert \vnormeps \Vert_{0,\Omega_2}:\epsilon>0\,\}$ is bounded, it follows that $\Vert \epsilon\,\vnormeps \Vert_{0,\Omega_2}=\epsilon \Vert \vnormeps \Vert_{0,\Omega_2}
\rightarrow 0$. Suppose first that $\vtan\neq0$ and consider the following quotients
\begin{equation*}
\frac{\Vert \vtaneps \Vert_{0,\Omega_2}}{\Vert \vnormeps \Vert_{0,\Omega_2}}=
\frac{\Vert \epsilon\,\vtaneps \Vert_{0,\Omega_2}}{\Vert \epsilon\,\vnormeps \Vert_{0,\Omega_2}}>
\frac{\Vert \vtan \Vert_{0,\,\Omega_2}-\delta}{\Vert \epsilon\,\vnormeps \Vert_{0,\Omega_2}}>0 .
\end{equation*}
The lower bound holds true for $\epsilon>0$ small enough and
adequate $\delta>0$. Then, we conclude that the ratio of tangent
component over normal component $L^{2}$-norms, blows-up to
infinity i.e., the tangential velocity is much faster than
the normal one in the thin channel.

If $\vtan=0$ we can not use the same reasoning, so a further
analysis has to be made. Suppose then that the solution
$[\,(\,\vone, \vtwo\,), \,(\,\pone, \ptwo\,)\,]$ of Problem \eqref{lower dimensional limit problem} is such that $\vtan = 0$. Then, the Equation
\eqref{lower dimensional limit problem 2} implies that $\vone\cdot \n = 0$ on $
\Gamma$ i.e., the problem on the region $\Omega_{1}$ is well-posed, independently from the activity on the interface $\Gamma$. The pressure on $\Gamma$ becomes subordinate and it must satisfy the following conditions:
\begin{equation*}
\ptwo = \big(\pone
-(\E+\alpha)\,\vone\cdot\n \big)\big\vert_{\Gamma}
= \big(\pone
-(\E+\alpha)\,\Q^{-1}\,\grad\,\pone\cdot\n \big)\big\vert_{\Gamma},
\end{equation*}
\begin{equation*}
\grad_{\!T}\,\ptwo = \mathbf{f}^{2}_{T} .
\end{equation*}
On the other hand, the values of $\pone$ are defined by $h^{1}$ then,
if we impose the condition on the forcing term $\mathbf{f}^{2}_T$
\begin{equation*}
\mathbf{f}^{2}_{\,T} \neq \grad_{\!T}\,\big(\,\pone -(\E+\alpha)\,\vone\cdot\n\,\big)\big\vert_{\Gamma} ,
\end{equation*}
we obtain a contradiction. Consequently, restrictions on the forcing terms $\mathbf{f}^{2}$ and $h^{1}$ can be given, so that  $\vtan \neq 0$ and the magnitudes relation $\Vert \vtaneps  \Vert_{0,\Omega_2}\gg
\Vert \vnormeps  \Vert_{0,\Omega_2}$ holds for $\epsilon>0$ small enough, as discussed above.
%
%
\bibliographystyle{plain}
\def\cprime{$'$} \def\cprime{$'$} \def\cprime{$'$}

\end{document}